\documentclass[11pt,reqno]{preprint}
\usepackage[full]{textcomp}
\usepackage[osf]{newtxtext}
\usepackage{comment}

\usepackage{amssymb}
\usepackage{mathtools}
\usepackage{hyperref}
\usepackage{breakurl}
\usepackage{mhenvs}
\usepackage{mhequ}
\usepackage{mhsymb}
\usepackage{booktabs}
\usepackage{tikz}
\usepackage{mathrsfs}
\usepackage{longtable}
\usepackage{halloweenmath}
\usepackage{scalerel}
\usepackage{cprotect}
\usepackage{multirow}
\usepackage{colortbl}
\usepackage{makecell}
\usepackage{resizegather}

\usepackage{wasysym}
\usepackage{centernot}
\usepackage{enumitem}
\usepackage{stackrel}

\usetikzlibrary{shapes.misc}
\usetikzlibrary{shapes.symbols}
\usetikzlibrary{shapes.geometric}
\usetikzlibrary{decorations}
\usetikzlibrary{decorations.markings}
\usetikzlibrary{patterns} 
\usetikzlibrary{snakes} 

\usepackage{subfig}
\usetikzlibrary{plotmarks}
\usetikzlibrary{decorations.pathreplacing}
\usetikzlibrary{decorations.pathmorphing}
\usetikzlibrary{calc}
\usetikzlibrary{external}
\usetikzlibrary{arrows.meta}

\newcommand{\Ceps}[1]{\mathcal{C}_{\eps}^{#1}}
\newcommand{\inner}[2]{\langle #1, #2 \rangle_{\eps}}
\newcommand{\rescaled}[3]{#1_{#2}^{#3}}

\newcommand\norm[1]{\left\lVert#1\right\rVert}

\let\oldskull\skull
\def\skull{\mathord{\oldskull}}

\DeclareMathAlphabet{\mathbbm}{U}{bbm}{m}{n}

\DeclareFontFamily{U}{BOONDOX-calo}{\skewchar\font=45 }
\DeclareFontShape{U}{BOONDOX-calo}{m}{n}{
  <-> s*[1.05] BOONDOX-r-calo}{}
\DeclareFontShape{U}{BOONDOX-calo}{b}{n}{
  <-> s*[1.05] BOONDOX-b-calo}{}
\DeclareMathAlphabet{\mcb}{U}{BOONDOX-calo}{m}{n}
\SetMathAlphabet{\mcb}{bold}{U}{BOONDOX-calo}{b}{n}
\DeclareMathAlphabet{\mathbcalboondox}{U}{BOONDOX-calo}{b}{n}

\setlist{noitemsep,topsep=4pt}

\makeatletter
\def\DeclareSymbol#1#2#3{%
	\expandafter\gdef\csname MH@symb@#1\endcsname{\tikzsetnextfilename{symbol#1}%
	\tikz[baseline=#2,scale=0.15,draw=symbols,line join=round,line cap=round]{#3}}%
	\expandafter\gdef\csname MH@symb@#1s\endcsname{\scalebox{0.75}{\tikzsetnextfilename{symbol#1}%
	\tikz[baseline=#2,scale=0.15,draw=symbols,line join=round,line cap=round]{#3}}}%
	\expandafter\gdef\csname MH@symb@#1ss\endcsname{\scalebox{0.65}{\tikzsetnextfilename{symbol#1}%
	\tikz[baseline=#2,scale=0.15,draw=symbols,line join=round,line cap=round]{#3}}}%
	}
\def\<#1>{\ifthenelse{\boolean{mmode}}{\mathchoice{\csname MH@symb@#1\endcsname}{\csname MH@symb@#1\endcsname}{\csname MH@symb@#1s\endcsname}{\csname MH@symb@#1ss\endcsname}}{\csname MH@symb@#1\endcsname}}
\makeatother

\makeatletter 
\newcommand*{\bigcdot}{}
\DeclareRobustCommand*{\bigcdot}{%
  \mathbin{\mathpalette\bigcdot@{}}%
}
\newcommand*{\bigcdot@scalefactor}{.5}
\newcommand*{\bigcdot@widthfactor}{1.15}
\newcommand*{\bigcdot@}[2]{%
  \sbox0{$#1\vcenter{}$}
  \sbox2{$#1\cdot\m@th$}%
  \hbox to \bigcdot@widthfactor\wd2{%
    \hfil
    \raise\ht0\hbox{%
      \scalebox{\bigcdot@scalefactor}{%
        \lower\ht0\hbox{$#1\bullet\m@th$}%
      }%
    }%
    \hfil
  }%
}
\makeatother

\newcommand{\cut}{\mathfrak{C}}

\newcommand{\mrd}{\mathop{}\!\mathrm{d}}

\newcommand{\mcC}{\mathcal{C}}

\newcommand{\mcS}{\mathcal{S}}

\newcommand{\T}{\mathbf{T}}

\newcommand{\bonds}{{\mathbf{B}}}
\newcommand{\obonds}{{\overline{\mathbf{B}}}}

\newcommand{\plaq}{{\mathbf{P}}}
\newcommand{\oplaq}{\overline{\mathbf{P}}}

\def\${|\!|\!|}

\def\id{\mathrm{id}}

\newcommand{\mfR}{\mathfrak{R}}

\newcommand{\mfp}{\mathfrak{p}}

\def\cC{\mathscr{C}}

\def\combplus[#1,#2,#3,#4]{\binom{#1\ {\scriptstyle #4} }{#2\ #3}}

\def\singlescalegenvert[#1,#2]{\hat{H}^{#2}_{#1}}
\def\multiscalegenvert[#1,#2]{H^{#2}_{#1}}

\def\nr[#1]{\tilde{N}[#1]} 
\def\inn[#1]{\mathring{N}[#1]}
\def\nrinn[#1]{\hat{N}_{#1}} 
\def\nrmod[#1,#2]{\tilde{N}_{#1}(#2)}
\def\nrinnmod[#1,#2]{\hat{N}_{#1}(#2)}

\def\ident[#1]{\underline{#1}}

\def\mylink#1#2{\mathrel{\vbox{\offinterlineskip\ialign{%
    \hfil##\hfil\cr
    $\scriptscriptstyle#1$\cr
    \noalign{\kern0.1ex}
    $#2$\cr
}}}}
\def\mysublink[#1]#2#3{\mathrel{\vbox{\offinterlineskip\ialign{%
    \hfil##\hfil\cr
    $\scriptscriptstyle#2$\cr
    \noalign{\kern0.1ex}
    $#3$\cr
    \noalign{\kern-0.2ex}
    \smash{\raisebox{-\height}{\hbox{$\scriptscriptstyle #1$}}}\cr
    \noalign{\kern0.2ex}
}}}}

\def\fon[#1]{\cC_{#1}}

\def\mincompproj[#1]{\mfp_{#1}}

\def\Proj_#1{\mathop{\mathrm{Proj}_{#1}}}

\def\negrenorm[#1]{\mfR_{#1}}
\def\topnegrenorm[#1]{\overline{\mfR}_{#1}}

\def\quotedge[#1]{E^{q}_{#1}}

\def\posrenorm[#1]{\mcC_{#1}}
\def\topposrenorm[#1]{\overline{\mcC_{#1}}}
\def\cutsmod[#1]{\mathbb{C}_{+,#1}}

\def\fullcutsmod[#1]{\cut_{#1}}

\colorlet{symbols}{blue!30!black!50}
\colorlet{testcolor}{green!60!black}
\colorlet{darkblue}{blue!60!black}
\colorlet{darkgreen}{green!60!black}
\definecolor{darkergreen}{rgb}{0.0, 0.5, 0.0}

\definecolor{purple}{rgb}{0.55,0.05,0.8}

\def\symbol#1{{\mathbf{#1}}}

\def\1{\mathbf{\symbol{1}}}

\makeatletter
\pgfdeclareshape{crosscircle}
{
 \inheritsavedanchors[from=circle] 
 \inheritanchorborder[from=circle]
 \inheritanchor[from=circle]{north}
 \inheritanchor[from=circle]{north west}
 \inheritanchor[from=circle]{north east}
 \inheritanchor[from=circle]{center}
 \inheritanchor[from=circle]{west}
 \inheritanchor[from=circle]{east}
 \inheritanchor[from=circle]{mid}
 \inheritanchor[from=circle]{mid west}
 \inheritanchor[from=circle]{mid east}
 \inheritanchor[from=circle]{base}
 \inheritanchor[from=circle]{base west}
 \inheritanchor[from=circle]{base east}
 \inheritanchor[from=circle]{south}
 \inheritanchor[from=circle]{south west}
 \inheritanchor[from=circle]{south east}
 \inheritbackgroundpath[from=circle]
 \foregroundpath{
   \centerpoint%
   \pgf@xc=\pgf@x%
   \pgf@yc=\pgf@y%
   \pgfutil@tempdima=\radius%
   \pgfmathsetlength{\pgf@xb}{\pgfkeysvalueof{/pgf/outer xsep}}%
   \pgfmathsetlength{\pgf@yb}{\pgfkeysvalueof{/pgf/outer ysep}}%
   \ifdim\pgf@xb<\pgf@yb%
     \advance\pgfutil@tempdima by-\pgf@yb%
   \else%
     \advance\pgfutil@tempdima by-\pgf@xb%
   \fi%
   \pgfpathmoveto{\pgfpointadd{\pgfqpoint{\pgf@xc}{\pgf@yc}}{\pgfqpoint{-0.707107\pgfutil@tempdima}{-0.707107\pgfutil@tempdima}}}
   \pgfpathlineto{\pgfpointadd{\pgfqpoint{\pgf@xc}{\pgf@yc}}{\pgfqpoint{0.707107\pgfutil@tempdima}{0.707107\pgfutil@tempdima}}}
   \pgfpathmoveto{\pgfpointadd{\pgfqpoint{\pgf@xc}{\pgf@yc}}{\pgfqpoint{-0.707107\pgfutil@tempdima}{0.707107\pgfutil@tempdima}}}
   \pgfpathlineto{\pgfpointadd{\pgfqpoint{\pgf@xc}{\pgf@yc}}{\pgfqpoint{0.707107\pgfutil@tempdima}{-0.707107\pgfutil@tempdima}}}
 }
}
\makeatother

\makeatletter
\pgfdeclareshape{crossrectangle}
{
 \inheritsavedanchors[from=rectangle] 
 \inheritanchorborder[from=rectangle]
 \inheritanchor[from=rectangle]{north}
 \inheritanchor[from=rectangle]{north west}
 \inheritanchor[from=rectangle]{north east}
 \inheritanchor[from=rectangle]{center}
 \inheritanchor[from=rectangle]{west}
 \inheritanchor[from=rectangle]{east}
 \inheritanchor[from=rectangle]{mid}
 \inheritanchor[from=rectangle]{mid west}
 \inheritanchor[from=rectangle]{mid east}
 \inheritanchor[from=rectangle]{base}
 \inheritanchor[from=rectangle]{base west}
 \inheritanchor[from=rectangle]{base east}
 \inheritanchor[from=rectangle]{south}
 \inheritanchor[from=rectangle]{south west}
 \inheritanchor[from=rectangle]{south east}
 \inheritbackgroundpath[from=rectangle]
 \foregroundpath{
   \centerpoint%
   \pgf@xc=\pgf@x%
   \pgf@yc=\pgf@y%
   \pgfutil@tempdima=\radius%
   \pgfmathsetlength{\pgf@xb}{\pgfkeysvalueof{/pgf/outer xsep}}%
   \pgfmathsetlength{\pgf@yb}{\pgfkeysvalueof{/pgf/outer ysep}}%
   \ifdim\pgf@xb<\pgf@yb%
     \advance\pgfutil@tempdima by-\pgf@yb%
   \else%
     \advance\pgfutil@tempdima by-\pgf@xb%
   \fi%
   \pgfpathmoveto{\pgfpointadd{\pgfqpoint{\pgf@xc}{\pgf@yc}}{\pgfqpoint{-0.707107\pgfutil@tempdima}{-0.707107\pgfutil@tempdima}}}
   \pgfpathlineto{\pgfpointadd{\pgfqpoint{\pgf@xc}{\pgf@yc}}{\pgfqpoint{0.707107\pgfutil@tempdima}{0.707107\pgfutil@tempdima}}}
   \pgfpathmoveto{\pgfpointadd{\pgfqpoint{\pgf@xc}{\pgf@yc}}{\pgfqpoint{-0.707107\pgfutil@tempdima}{0.707107\pgfutil@tempdima}}}
   \pgfpathlineto{\pgfpointadd{\pgfqpoint{\pgf@xc}{\pgf@yc}}{\pgfqpoint{0.707107\pgfutil@tempdima}{-0.707107\pgfutil@tempdima}}}
 }
}
\makeatother

\colorlet{greennode}{green!50!black}
\colorlet{rednode}{red!50!black}
\colorlet{lbluenode}{blue!25}
\colorlet{dbluenode}{blue}
\colorlet{orangenode}{orange}

\definecolor{connection}{rgb}{0.7,0.1,0.1}

\tikzset{
dot/.style={circle,fill=black,inner sep=0pt, minimum size=1mm},
root/.style={circle,fill=black!50,inner sep=0pt, minimum size=3mm},
       var/.style={circle,fill=black!10,draw=black,inner sep=0pt, minimum size=1.6mm},
       delta/.style={densely dotted},
       var1/.style={rectangle,fill=black!10,draw=black,inner sep=0pt, minimum size=1.6mm},
       var2/.style={diamond,fill=black!10,draw=black,inner sep=0pt, minimum size=2mm},
       kernel/.style={semithick,shorten >=2pt,shorten <=2pt},
      kernel1/.style={postaction={decorate,decoration={markings,mark=at position 0.45 with {\draw[-] (0,-0.08) -- (0,0.08);}}}},
       kernels/.style={snake=snake,segment amplitude=1pt,segment length=4pt},
       rho/.style={densely dashed,semithick,shorten >=2pt,shorten <=2pt},
          testfcn/.style={dotted,semithick,shorten >=2pt,shorten <=2pt},
          tau/.style={circle,inner sep=1pt,draw=black,fill=white,text=black,thin},
       renorm/.style={shape=circle,fill=white,inner sep=1pt},
       labl/.style={shape=rectangle,fill=white,inner sep=1pt},
       xi/.style={very thin,circle,fill=lbluenode,draw=symbols,inner sep=0pt,minimum size=1.2mm},
       xi1/.style={very thin,rectangle,fill=lbluenode,draw=symbols,inner sep=0pt,minimum size=1.2mm},
       xi2/.style={very thin,diamond,fill=lbluenode,draw=symbols,inner sep=0pt,minimum size=1.6mm},
       xigreen/.style={very thin,circle,fill=greennode,draw=symbols,inner sep=0pt,minimum size=1.2mm},
       xigreen1/.style={very thin,rectangle,fill=greennode,draw=symbols,inner sep=0pt,minimum size=1.2mm},
       xired/.style={very thin,circle,fill=rednode,draw=symbols,inner sep=0pt,minimum size=1.2mm},
       xilblue/.style={very thin,circle,fill=lbluenode,draw=symbols,inner sep=0pt,minimum size=1.2mm},
       xidblue/.style={very thin,circle,fill=dbluenode,draw=symbols,inner sep=0pt,minimum size=1.2mm},
       xiorange/.style={very thin,circle,fill=orangenode,draw=symbols,inner sep=0pt,minimum size=1.2mm},
       xix/.style={crosscircle,fill=lbluenode,draw=symbols,inner sep=0pt,minimum size=1.2mm},
xix-green-red/.style={circle, fill=greennode!70!white,draw=rednode,inner sep=0pt,minimum size=1.6mm,append after command={node [inner sep=0pt,minimum size=0.8mm,thick, draw = rednode, cross out]{}}},
xix-green-red1/.style={rectangle, fill=greennode!70!white,draw=rednode,inner sep=0pt,minimum size=1.5mm,append after command={node [inner sep=0pt,minimum size=1mm,thick, draw = rednode, cross out]{}}},
	xib/.style={very thin,circle,fill=lbluenode,draw=symbols,inner sep=0pt,minimum size=1.6mm},
	xib1/.style={very thin,rectangle,fill=lbluenode,draw=symbols,inner sep=0pt,minimum size=1.6mm},
	xie/.style={very thin,circle,fill=greennode,draw=symbols,inner sep=0pt,minimum size=1.6mm},
	xid/.style={very thin,circle,fill=lbluenode,draw=symbols,inner sep=0pt,minimum size=1.6mm},
	xibx/.style={crosscircle,fill=lbluenode,draw=symbols,inner sep=0pt,minimum size=1.6mm},
	kernels2/.style={ultra thick,draw=symbols,segment length=12pt},
	not/.style={thin,regular polygon, regular polygon sides=3,draw=connection,fill=connection,inner sep=0pt,minimum size=1.2mm},
	notlblue/.style={thin,regular polygon, regular polygon sides=3,draw=lbluenode,fill=lbluenode,inner sep=0pt,minimum size=1.2mm},
	notorange/.style={thin,regular polygon, regular polygon sides=3,draw=orangenode,fill=orangenode,inner sep=0pt,minimum size=1.2mm},
	notgreen/.style={thin,regular polygon, regular polygon sides=3,draw=greennode,fill=greennode,inner sep=0pt,minimum size=1.2mm},
	>=stealth,
 }

\newtheorem{assumption}[lemma]{Assumption}

\colorlet{darkblue}{blue!90!black}
\colorlet{darkred}{red!90!black}
\colorlet{darkgreen}{green!70!black}
\let\comm\comment

\def\hao#1{\comm[darkred]{HS: #1}}
\def\yahui#1{\comm[blue]{YQ: #1}}

\def\s{\mathfrak{s}}

\newcommand{\e}{\varepsilon}

\def\${|\!|\!|}
\def\Wick#1{\colon\!\! #1 \! \colon}

\def\?{{\color{red}?}}

\def\id{\mathrm{id}}

\def\id{\mathrm{id}}

\def\dash{\leavevmode\unskip\kern0.18em--\penalty\exhyphenpenalty\kern0.18em}
\def\slash{\leavevmode\unskip\kern0.15em/\penalty\exhyphenpenalty\kern0.15em}

\usepackage{stmaryrd}

\newtheorem{example}[lemma]{Example}

\long\def\yahuiText#1{{\color{darkgreen}Yahui:\ #1}}

\let\basepoint\logof
\def\logof{\mathord{{\basepoint}}} 

\title{Scaling limit of the 3D abelian Yang--Mills Langevin dynamics}
\author{Ilya Chevyrev$^1$, Yahui Qu$^2$, Hao Shen$^3$}

\institute{SISSA, Trieste, Italy \and
		University of Wisconsin-Madison, Madison, USA \and University of Chicago, Chicago, USA\\
		\email{ichevyrev@gmail.com, yqu45@wisc.edu, \\pkushenhao@gmail.com}}

\begin{document}
\maketitle
\begin{abstract}
We study the continuum scaling limit of the Langevin dynamics for three-dimensional U(1) lattice Yang--Mills theory. The model is defined on the discrete 3D torus   with a general class of plaquette actions that are suitably normalized, including Wilson, Manton, and Villain actions. Under the weak-coupling scaling and in the DeTurck gauge, we prove that, the rescaled logarithmic field converges to the solution of the one-form stochastic heat equation in probability. In particular, the limiting dynamics are universal and do not depend on the higher-order details of the plaquette action.
\\[.4em]
\noindent {\small\textit{MSC 2020 classification:} 60H15 (Primary), 60L30, 81T13, 81T27 (Secondary)}
\end{abstract}
\setcounter{tocdepth}{2}

\tableofcontents

\section{Introduction}
\label{sec:Intro}

In this paper, we study the scaling limit of the Langevin dynamics
associated with the three-dimensional $U(1)$ lattice Yang--Mills models.
In the continuum setting, given a Lie group $G$, the Euclidean Yang--Mills action functional for a connection $A$ is given by $S(A)=\frac12 \int |F_A|^2 dx$. Here, the connection $A$ is a $1$-form taking values in the associated Lie algebra, and $F_A=dA+A\wedge A$ is 
the curvature $2$-form. 

In the abelian case, where $G=U(1)$,  the connection $A$ is
 a real-valued $1$-form, and the curvature simplifies to 
$F_A=dA$.
The action then becomes quadratic $S(A)=\frac12 \int |dA|^2 dx$.
In this case, the formal Yang--Mills measure is $e^{-S(A)} DA$, which can be interpreted as a Gaussian measure under a suitable gauge. 
A fundamental property of this theory is its invariance under gauge transformations,
namely, mapping  $A\mapsto A+d\phi$
for a scalar function $\phi$ leaves the curvature $dA$, and thus the action, unchanged.

Lattice gauge theory is the discretization of the Yang-Mills model
which preserves the gauge symmetry at the discrete level. The connection $A$ is replaced by a field defined on the lattice edges, which takes values in the Lie group $G$.
On a finite lattice, the corresponding lattice Gibbs measure is always well-defined.

A dynamic approach to constructing the continuum limits of such theories is provided by stochastic quantization. This method defines a Langevin equation whose invariant measure formally coincides with the target Euclidean quantum field theory measure. In recent years, the stochastic quantization of the Yang--Mills measure has been successfully constructed in two and three dimensions
\cite{CCHS_2D,CCHS_3D}  (see also \cite{Chevyrev22YM}). Furthermore, the convergence of both the measures and the dynamics of a large class of discrete lattice gauge theories to their continuum limits has been established in the two-dimensional setting \cite{chevyrev2023invariant}.

It is an open problem to extend the convergence result \cite{chevyrev2023invariant} from two to three dimensions.
The three-dimensional setting presents profound analytical challenges since it is much more singular.
In this paper, we make the first step towards this goal by restricting our attention to the abelian $U(1)$ case. 

We now introduce the rigorous setup.
We consider lattice gauge theory with abelian Lie group $U(1)$, defined on a periodic lattice $\Lambda_L = \Z^d/(L\Z^d)$ of size $L = 2^{N}$, $N\geq 1$.
We denote by $\bonds$ the set of oriented edges; $\obonds$, the positively oriented edges; and $\plaq$ the plaquettes. We refer to Section~\ref{sec:Notation} for precise definitions of these notations.

For an oriented edge $e=(x,y)$, write $\bar e=(y,x)$ for the same edge with opposite orientation. A compact $U(1)$-connection is a map
\[
    U:\bonds\to U(1)
\]
satisfying the orientation convention
$ U(\bar e)=U(e)^{-1}$.
Equivalently, writing $U(e)=\exp(i\vartheta(e))$, we regard $\vartheta$ as an $\R/(2\pi\Z)$-valued $1$-form satisfying
\[
    \vartheta(\bar e)=-\vartheta(e)
    \qquad \text{in } \R/(2\pi\Z).
\]
Here $-\vartheta(e)$ denotes the additive inverse of the class $\vartheta(e)$ in the quotient group $\R/(2\pi\Z)$. When a real representative is needed, we choose the principal representative on $\obonds$ and extend to all oriented edges by the above antisymmetry convention.

We identify functions on the circle $U(1)$ with $2\pi$-periodic functions on the real line $\R$.
Let $s \colon U(1) \to \R\cup\{\infty\}$ be a measurable function such that $s(x)=s(x^{-1})$. 
Equivalently, we can write $s(e^{i\alpha})=h(\alpha)$ for a $2\pi$-periodic even function $h$.
(See Section~\ref{sec:theorem} for the precise conditions on $s$ and $h$.)
For a plaquette $p=(e^{(1)},e^{(2)},e^{(3)},e^{(4)})$, define
\[
 U(\partial p) = U(e^{(1)}) U(e^{(2)}) U(e^{(3)}) U(e^{(4)})
\]
and
\[
\vartheta(\partial p)
= \vartheta(e^{(1)})+\vartheta(e^{(2)})+\vartheta(e^{(3)})+\vartheta(e^{(4)})
\qquad \text{in } \R/(2\pi\Z).
\]
Since $h$ is $2\pi$-periodic, $h[\vartheta(\partial p)]$ is well-defined, independently of the choice of representatives.

Let $\mathfrak{G}=\mathrm{Map}(\Lambda,U(1))$ denote the group of gauge transformations. It acts on lattice connections by
\[
U^g(x,y) = g(x) U(x,y) g(y)^{-1},
\qquad g\in \mathfrak{G}.
\]
This action preserves the orientation convention $U(\bar e)=U(e)^{-1}$. For a lattice connection $U$, we write
\[
    [U]=\{U^g:g\in\mathfrak{G}\}
\]
for its gauge orbit. We say $\bar U$ and $U$ are gauge equivalent and write $\bar U\sim U$ if $\bar U\in[U]$.
If $g(x)=\exp(i\phi(x))$, then in compact angle variables
\[
    \vartheta^g(x,y)=\vartheta(x,y)+\phi(x)-\phi(y)
    \qquad \text{in } \R/(2\pi\Z).
\]
Notice that $U(\partial p)$, and thus the action functional, is invariant under gauge transformations:
\[
U^g(\partial p) = g(x)U(\partial p ) g(x)^{-1}.
\]
In the abelian case, this is simply $U^g(\partial p)=U(\partial p)$.

Define the action functional of the $U(1)$ lattice gauge theory by
\begin{equ}[e:def-CSU]
\mcS(U)
=
 \frac14 \sum_{p\in\plaq} s[U(\partial p)]
=
 \frac14\sum_{p\in\plaq} h[\vartheta(\partial p)]\;.
\end{equ}
Let $\beta>0$ denote the lattice inverse coupling. Consider the probability measure $\mu_{N,\beta}$ on compact lattice connections, identified with $U(1)^{\obonds}$ by restriction to positively oriented edges, defined by
\begin{equ}\label{eq:mu_N}
\mu_{N,\beta}(\mrd U) \eqdef \tilde{Z}_{\beta}^{-1} e^{-\beta \mcS(U)} \mrd U\;,
\end{equ}
where $\mrd U$ is the Haar (i.e., uniform) measure on $ U(1)^{\obonds} $ and
 $\tilde Z_{\beta}$ is the normalization constant which makes $\mu_{N,\beta}$ a probability measure
 ($\mu_{N,\beta}$ and $\tilde{Z}_{\beta}$ are well-defined whenever $s$ is bounded).
 
 The above $U(1)$ lattice gauge theory, in particular in 3D, has a rich history. Early foundational work by Gross \cite{MR728862} established the convergence of the model to a Gaussian continuum limit. 
 Specifically, for a class of models (see Remark~\ref{rem:Gross-assumptions}), Gross showed that the discrete curvature $h'(\vartheta(\partial p))$,
 which can be decomposed into an exact part (``electric sector'') and a co-exact part (``magnetic sector''), when projected to the first sector, converges to a free Gaussian field.
The scaling in \cite{MR728862} is the same as ours in \eqref{e:scaling-x} below, 
written as $\beta^{-1}=\eps g^2$, where $g$ is fixed.

Varying the single-plaquette function $s$, or equivalently $h$, in \eqref{eq:mu_N} defines a class of models.
We give some canonical examples  (see also \cite[Section~2]{MR728862}):
\begin{itemize}
\item
\textbf{Wilson model:} In general, for the $U(N)$ model, the Wilson model is defined by choosing the function $s$ on $U(N)$ as 
$s(x)=  \Re \mbox{Tr}(\id-x)$. 
In the $U(1)$ case, one has
\begin{equ}[e:Wilson-h]
s(x)= \Re (1-x)=  \Re (1-e^{i\alpha}),
\qquad
\mbox{thus}
\quad
 h(\alpha)= 1-\cos(\alpha).
\end{equ}
More generally, we can consider
\begin{equ}[e:Wilson-h-gen]
s(x)= m^{-2}\Re (1-x^m)=  m^{-2} \Re (1-e^{im \alpha}),
\end{equ}
namely,
 $h(\alpha) = m^{-2} (1-\cos(m\alpha))$ for $m=1,2,3,\cdots$, or any convex finite sum of these.
Note that the factor $m^{-2} $ is such that $h''(0)=1$.

\item
\textbf{Manton model: }
For general Lie group $G$ endowed with distance $|\,\cdot\,|_G$, the Manton model is defined, up to the normalization of the Yang--Mills action used here, by the squared distance to the identity. In the $U(1)$ case we take
\[
h(\alpha)
= \frac12 \min_{k\in \Z} |\alpha-2\pi k|^2\;.
\]
In particular, when $|\alpha|<\pi$, one has $h(\alpha)=\frac12\alpha^2$, and hence $h''(0)=1$. Note that in this case $h$ is not differentiable at $(2k+1)\pi$ for $k\in \Z$.
\item 
\textbf{Villain model:}
 We define  an even, $2\pi$-periodic smooth function
$h=h_\lambda$ for $\lambda \geq 1$ by
\begin{equ}[e:Villain-h]
\exp\{-c_\lambda\lambda h_\lambda(\alpha)\}
    =
    \sum_{n\in\Z}
    \exp\left\{
        -\frac{\lambda}{2}(\alpha-2\pi n)^2
    \right\},
\end{equ}
where $c_\lambda>0$ is chosen so that
\[
    h_\lambda''(0)=1.
\]
More explicitly, as $\lambda\to\infty$,
\[
    c_\lambda
   =
    1
    -
    \lambda
    \frac{
        \sum_{n\in\Z}
        (2\pi n)^2 e^{-2\pi^2\lambda n^2}
    }{
        \sum_{n\in\Z}
        e^{-2\pi^2\lambda n^2}
    } 
     =
    1-8\pi^2\lambda e^{-2\pi^2\lambda}
    +O\bigl(\lambda e^{-4\pi^2\lambda}\bigr).
\]
Thus, under the scaling \eqref{e:scaling-theta} below and taking
$ c_{\lambda}\lambda=2\beta$
in the Gibbs convention \eqref{eq:mu_N}, one has $\lambda =  2\beta + O(\beta^2 e^{-4\pi^2\beta})$.
\end{itemize}

\begin{remark}
Alternatively, if one takes the more naive choice $\exp \{-2\beta h_{\beta}(\alpha)\} = \sum_{n\in\Z}
    e^{ -\beta (\alpha-2\pi n)^2}$ in place of \eqref{e:Villain-h},
then the second derivative of $h_\beta$ at the origin is
$c_{2\beta}$ rather than exactly $1$. Since
$ c_{2\beta}-1  =  O(\beta e^{-4\pi^2\beta})$,
this only introduces exponentially small perturbations in the linear
coefficient and in the higher Taylor coefficients of the Langevin equation.
In dimension $3$, we will use the scaling
 $ \beta =\bar\beta\,\eps^{-1}$, so
these errors are smaller than any power of $\eps$. They could therefore
be absorbed in the estimates below and we would obtain the same limit.
\end{remark}

We study the scaling limits of the dynamics of these models.\footnote{Here we rescale the large lattice $\Lambda_L$ to $\T^d_\e$, which is close to the setup in \cite{MR728862}, but slightly different from  \cite{chevyrev2023invariant} which simply starts with discrete models on $\T^d_\e$.}
Assuming that the dimension $d$ is arbitrary for the moment, we rescale the underlying space into a unit torus of lattice spacing $\e>0$: 
\begin{equ}[e:scaling-x]
\Lambda_L \to \T^d_\e ,\qquad x\to x/L
\end{equ}
where $\e=L^{-1}$, see 
Section~\ref{sec:Notation} for precise definitions of these notations.
Moreover, we choose the canonical weak-coupling scaling of the lattice inverse coupling and work in a local chart near the identity:
\begin{equ}[e:scaling-theta]
\beta=\beta_\e \eqdef \bar\beta\,\e^{d-4},
\qquad
\vartheta=\e\theta\;,
\qquad
e^{i\vartheta}=e^{i\e\theta},
\end{equ}
where $\bar\beta>0$ is fixed. Here $\theta$ is a real-valued $1$-form in the chosen chart; on positively oriented edges it takes values in $(-\pi/\e, \pi/\e]$. \footnote{So on negatively oriented edges $\theta$ takes values in $[-\pi/\e, \pi/\e)$. The endpoints of this interval will play no role, since in the analysis of the dynamics we work up to the stopping time in a smaller interval.}
For a real-valued $1$-form $\theta$ on the oriented edges of $\T^d_\e$, we define the ``discrete exterior differential'' (see also Section~\ref{sec:Notation}) by
\[
\mrd_\e \theta(p) = \e^{-1} \theta(\partial p).
\]

Heuristically, under the above scaling, the Manton model has Gibbs weight
\[
\exp\Big\{
-\bar\beta\e^{d-4}
\sum_{p\in\oplaq_\e}
\e^2 (\theta(\partial p))^2
\Big\}
\approx
\exp\Big\{
-\bar\beta \int_{\T^d}
 |\mrd_\e \theta|^2 \,dx
\Big\},
\]
where $\oplaq_\e$ is the set of positively oriented plaquettes and
$\e^d\sum_{p\in\oplaq_\e}$ is replaced by the corresponding continuum integral. Similarly, the Wilson model has Gibbs weight
\begin{equs}{}
&\exp\Big\{
-2\bar\beta\e^{d-4} \sum_{p\in\oplaq_\e}
\Big(
\frac{\eps^2}2 (\theta(\partial p))^2
-\frac{\eps^4}{4!} (\theta(\partial p))^4
+\frac{\eps^6}{6!} (\theta(\partial p))^6
+\cdots
\Big)
\Big\}
\\
&
\approx
\exp\Big\{
-2\bar\beta
\sum_{1\leq i < j \leq d}\int_{\T^d}
\Big(
 \frac12 |(\mrd_\e \theta)_{ij}|^2
 - \frac{\e^4}{4!} |(\mrd_\e\theta)_{ij}|^4
+   \frac{\e^8}{6!} |(\mrd_\e\theta)_{ij}|^6
 +\cdots \Big) dx
\Big\} \;.
\end{equs}

The fixed constant $\bar\beta$ only changes the coefficient of the limiting gauge-fixed heat operator. 
To make the notation for the rest of the paper less heavy, we normalize
\begin{equation}\label{eq:barbeta1}
    \bar\beta=1.
\end{equation}

Note that the leading term is the quadratic term $|\mrd_\e\theta|^2$.
This indicates that the corresponding lattice dynamics should converge to 
the limit (after suitable gauge fixing) given by the
stochastic heat equation 
\begin{equ}[e:SHE]
\partial_t \theta = \Delta \theta + \xi \;.
\end{equ}

Our main result is that this is indeed true, see Theorem~\ref{thm:main}.
Let us emphasize that,
even though the stochastic heat equation is well-posed in all dimensions,
 the dimension $d$ still plays a crucial role.
Indeed, suppose that after suitable gauge fixing $\theta$ has the regularity of the Gaussian free field, namely $\CC^{-(d-2)/2-}$. Assigning (parabolic) homogeneity $-1$ to the derivative $\mrd_\e$ and $1$ to a factor of $\e$, the terms in the Wilson action have the formal homogeneities
\begin{equ}[e:scaling-meas]
| \mrd_\e \theta |^2 \sim -d,
\qquad
\e^4 |\mrd_\e \theta|^4 \sim -2d+4,
\qquad
\e^8 |\mrd_\e\theta|^6 \sim -3d+8, \quad \mbox{etc.}
\end{equ}
We immediately see that if $d<4$, the above numbers are increasing, 
and if $d=4$, they stay at $-4$,
and if $d>4$, they are decreasing.
In other words, we do have a ``subcritical'' regime $d<4$
for the scaling limit problem.

In the next section, we state our main theorem. To that end, we first introduce our notation precisely.

\section{Main result}

\subsection{Notation}
\label{sec:Notation}


Let $\T^d=\R^d/\Z^d$ denote the $d$-dimensional torus, which we identify as a set with $[0,1)^d$.
Let $\T^d_\e \subset \T^d$
denote the discrete torus with lattice spacing $\e=2^{-N}$ for some $N\ge 1$.

We denote $[d]=\{1,\ldots, d\}$ and write $\{e_j\}_{j\in [d]}$ for the
set of unit basis vectors in $\R^d$.
For $j\in [d]$, we will write 
$\eps_j$ as shorthand for $\eps e_j$.
Let
\[
\bonds_\e \eqdef \{(x,x\pm \e_j)\in \T^d_\e\times \T^d_\e \,:\, x\in\T^d_\e\,,\, j\in [d]\}
\]
denote the set of oriented edges on $\T^d_\e$.
For $e=(x,y)\in\bonds_\e$, write $\bar e=(y,x)$ for the same edge with the opposite orientation.
We also write 
$\obonds_{\e,j} \eqdef \{(x,x+\eps_j)\,:\, x\in\T^d_\e\}$
and $\obonds_\e = \cup_{j\in[d]} \obonds_{\e,j}$ 
for the set of positively oriented edges.

An oriented plaquette (or simply a plaquette) is a tuple $p=(x,a,b)$, where $x\in \T^d_\e$ is a lattice site, 
$a,b \in \{\pm\e_j\}_{j\in [d]}$
where $a\neq \pm b$. 
We will also use the alternative notation 
\begin{equ}
p=(x_1,x_2,x_3,x_4)\eqdef (x,x+a,x+a+b,x+b)
\end{equ}
or
$p=(e^{(1)},e^{(2)},e^{(3)},e^{(4)})$ where $e^{(i)}=(x_i,x_{i+1})$ for $i=1,\ldots, 4$ with $x_5\eqdef x_1$.
We will call $x$ the base point of $p$, and say that $\overleftarrow{p}\eqdef (x,b,a)$ is the plaquette with reverse orientation.
We denote by $\plaq\equiv \plaq_\e$ the set of all plaquettes.
We also define $\oplaq_\e \subset \plaq_\e$ to be  the set of plaquettes of the form 
$p=(x,\e_i,\e_j)$, where $x\in \T^d_\e$ and $1\le i < j \le d$.
We have $|\plaq_\e|=8|\oplaq_\e|$.

\medskip

Let us also define the pre-scaling lattice as follows.

For $L=2^N\geq 1$, let $\Lambda_L$ denote the discrete $d$-dimensional torus with unit spacing 
which we identify as a set with $\{0,\ldots, L-1\}^d$.
Let $\bonds_L$, $\obonds_{L,j}$, $\obonds_L$ be  
defined as above but with $\eps$ replaced by $1$ and $\T^d_\e$ replaced by 
$\Lambda_L$.
Let $\plaq_L$ be the set of all plaquettes for $\Lambda_L$, defined as above
but with $x\in \Lambda_L$ and $\e=1$, and let $\oplaq_L\subset\plaq_L$ be the corresponding set of positively oriented plaquettes.

\medskip

We define $0$-forms to be functions $f\colon \T^d_\e \to \R$, and refer to them simply as ``functions''. A real-valued $1$-form is a function $\theta\colon\bonds_\e\to\R$ satisfying
\[
    \theta(\bar e)=-\theta(e).
\]
Equivalently, it is determined by its components on positively oriented edges,
\[
    \theta_j(x)\eqdef \theta(x,x+\eps_j),
    \qquad j\in[d].
\]
A compact angle field is an $\R/(2\pi\Z)$-valued $1$-form $\vartheta$ satisfying the same antisymmetry convention in the quotient group,
\[
    \vartheta(\bar e)=-\vartheta(e)
    \qquad \text{in } \R/(2\pi\Z).
\]
A real-valued $2$-form is a function $F\colon \plaq_\e \to \R$ which is invariant
under cyclic permutations of the vertices of a plaquette and changes
sign under reversal of orientation, i.e., $F(\overleftarrow{p})=-F(p)$.
In particular, $F$ is uniquely
determined by its values on $\oplaq_\e$.

For functions $f_1,f_2$, $1$-forms $\theta_1,\theta_2$, and $2$-forms $F_1,F_2$,
define their $L^2$ inner products
\begin{align}
\label{def: innerProdDisct}
\langle f_1 ,f_2 \rangle_\e^{(0)}
=\e^d \sum_{x\in \T^d_\e} f_1(x)f_2(x),
\quad
\langle \theta_1 ,\theta_2 \rangle_\e^{(1)}
= \e^d\sum_{e\in \obonds_\e} \theta_1(e)\theta_2(e),
\end{align}
where the superscripts indicate that the inner product is for $k$-forms for $k=0,1,2$,
and
\[
\langle F_1 ,F_2 \rangle_\e^{(2)}
=\e^d \sum_{p\in \oplaq_\e} F_1(p)F_2(p).
\] 
In the following sections, we often drop the superscript for simplicity of notation. 

For a function $f$ and a $1$-form $\theta$,
define, for $e=(x,y)$ and $p=(e^{(1)},e^{(2)},e^{(3)},e^{(4)})$,
\[
\mrd_\e f (e)=\e^{-1} (f(y)-f(x)),
\qquad
\mrd_\e \theta (p) = \e^{-1} \theta(\partial p),
\]
where
\[
    \theta(\partial p)
    \eqdef \theta(e^{(1)})+\theta(e^{(2)})+\theta(e^{(3)})+\theta(e^{(4)}).
\]
These are a $1$-form and a $2$-form, respectively.
For a $1$-form $\theta$ and a $2$-form $F$,
define
\[
\mrd_\e^* \theta (x)\eqdef \e^{-1}   \sum_{j=1}^d
\Big(  \theta(x-\e_j,x)- \theta(x,x+\e_j) \Big),
\]
\begin{equ}[e:d*F]
\mrd_\e^* F (x,x+\e_i) = \e^{-1}  \sum_{j\in [d],j\neq i} 
\Big( F(\bar p_j) + F(p_j) \Big)
\end{equ}
where $\bar p_j= (x,x+\e_i,x+\e_i+\e_j,x+\e_j)$
and $p_j =(x,x+\e_i,x+\e_i-\e_j,x-\e_j)$.

One can check that $\mrd_\e \mrd_\e f=0$,
and $\mrd_\e^* \mrd_\e^* F=0$.
 Moreover, $\mrd_\e^*$ is the adjoint of $\mrd_\e$, namely,
\begin{equ}[e:adjoint-d]
\langle \mrd_\e A , B \rangle_\e^{(k)} = \langle A, \mrd_\e^* B\rangle_\e^{(k-1)}\;,
\qquad
(k=1,2)
\end{equ}
where $A$ is a $0$-form and $B$ is a $1$-form,
or $A$ is a $1$-form and $B$ is a $2$-form.


Finally, for a $1$-form $\theta$, we define 
\begin{equ}[e:def-Laplace]
\Delta_\e \theta = -(\mrd_\e \mrd_\e^*+\mrd_\e^*\mrd_\e)
\end{equ}
and one can check that on each component it is the usual finite difference Laplacian, namely,
for $e\in \obonds_{\e,i}$ with $i\in [d]$,
\[
\Delta_\e \theta(e) = \e^{-2} \sum_{j=1}^d \Big( \theta(e+\e_j)+\theta(e-\e_j) - 2\theta(e)\Big).
\]
In particular, standard heat kernel estimates for $e^{t\Delta_\e}$ hold.

\begin{remark}
One can define $\mrd_\e$ and $\mrd_\e^* $ on general $k$-forms
(see e.g., \cite{MR728862}). This is not necessary for our purposes, so we do not give 
this general definition.
\end{remark}

\subsection{Main theorem}
\label{sec:theorem}

Let $\xi=(\xi_1,\ldots,\xi_d)$ be a $d$-tuple of independent space-time white noises.
For $e=(x,x+\e_j)\in\obonds_{\e,j}$, denote by $m(e)=x+\frac12\e_j$ its midpoint and define the half-open cube
\[
B(e,\e)
=
 m(e)+[-\e/2,\e/2)^d
\subset \T^d.
\]

For each $j$, the cubes $\{B(e,\e)\,:\,e\in\obonds_{\e,j}\}$ form a partition of $\T^d$. We define the discrete white noise by
\begin{align}
\label{eq:Disc_white_noise}
\xi^\e(t,e) = \e^{-d} \langle \xi_j(t,\cdot), \mathbf{1}_{B(e,\e)}\rangle,
\qquad e\in\obonds_{\e,j}.
\end{align}
Here, the pairing is the usual distributional pairing on $\T^d$. This normalization preserves the Itô isometry for the inner product \eqref{def: innerProdDisct}:
\begin{align*}
\E[\inner{f_1}{\xi^\eps(t)}^{(1)} \inner{f_2}{\xi^\eps(s)}^{(1)}] = \delta(t-s)\inner{f_1}{f_2}^{(1)}
\end{align*}
for discrete $1$-forms $f_1,f_2$ on $\T^d_\e$.

Since the Haar measure on $U(1)$ corresponds to the uniform measure on $(-\pi,\pi]$ after choosing principal representatives on positively oriented edges, we rewrite \eqref{eq:mu_N} as
\begin{equation}
	\label{eq:mu_Nvartheta}
	\mu_{N,\beta}(d \vartheta) \eqdef
	Z_{\beta}^{-1} 
	e^{- 2\beta\sum_{p\in \oplaq_L} h(\vartheta(\partial p ))} \prod_{e\in \obonds_L} d \vartheta(e)\;,
\end{equation}
where $\vartheta$ is represented on $\obonds_L$ by principal angles in $(-\pi,\pi]$ and extended to all oriented edges by $\vartheta(\bar e)=-\vartheta(e)$ in $\R/(2\pi\Z)$.
(All eight orientations of each geometric plaquette occur in $\plaq$. Since $h$ is even, restricting the sum in \eqref{e:def-CSU} to one positively oriented plaquette per geometric square produces an overall factor $2$ in place of $\frac14$.)

We first specify the class of single-plaquette actions $h$ considered in this paper. Besides being $2\pi$-periodic and even as mentioned in Section~\ref{sec:Intro}, we impose the following local assumptions near the identity.


\begin{assumption}
\label{assmp:YMAction}
Let $I$ 
be a nonempty index set. 
For each $\lambda\in I$, let
$h_\lambda\colon\R\to\R$ be an even $2\pi$-periodic function.
We assume that there exist constants $\rho_0\in(0,\pi)$ and $C<\infty$,
independent of $\lambda$, such that $h_\lambda\in \CC^6([-\rho_0,\rho_0])$,
$  h_\lambda''(0)=1$
and
\[
    \sup_{\lambda\in I}\sup_{|\alpha|\le \rho_0}
    \max_{0\le k\le 6}
    |h_\lambda^{(k)}(\alpha)|
    \le C.
\]
\end{assumption}

For the rest of the paper, to lighten the notation, when $I$ is a singleton or when the dependence on $\lambda$ is unimportant, we simply write $h=h_\lambda$.

The above assumptions imply that $h'(0)=0$  and $h(0)$ is a local minimum.
In the three examples mentioned in Section \ref{sec:Intro}, we take $I$ to be an (arbitrary) singleton for the Wilson and Manton actions, and remark that Assumption~\ref{assmp:YMAction} is clearly satisfied;
a non-trivial family of single-plaquette actions $h_\lambda$
is needed only for the Villain action.
(For the Manton action, smoothness only holds in a neighborhood of $0$; globally it fails at $(2k+1)\pi$, $k\in\Z$.)

\begin{lemma}\label{lem:villain-assumption}
For the Villain model, there exists $\lambda_0>0$ sufficiently large such that, with $I=[\lambda_0,\infty)$, the family $\{h_\lambda\}_{\lambda\in I}$ defined in
\eqref{e:Villain-h} satisfies Assumption~\ref{assmp:YMAction}.
\end{lemma}

\begin{proof}
Fix $\rho_0\in(0,\pi)$. Write $W_\lambda(\alpha)$ for the right-hand side of \eqref{e:Villain-h}, namely,
\[
    h_\lambda(\alpha)
    =
    -\frac{1}{c_\lambda\lambda}\log W_\lambda(\alpha).
\]
The function $W_\lambda$ is smooth, even and $2\pi$-periodic, and hence the same is true for $h_\lambda$. Also by definition $h_\lambda''(0)=1$.
It remains to check that the local $\CC^6$ bounds can be chosen uniformly for all large $\lambda$.
On $|\alpha|\le \rho_0$, factor out the $n=0$ term in \eqref{e:Villain-h} and set
\[
    r_\lambda(\alpha)
    \eqdef
    e^{\lambda\alpha^2/2}W_\lambda(\alpha)-1
    =
    \sum_{n\neq 0}
    \exp\left\{
        -\frac{\lambda}{2}\bigl((\alpha-2\pi n)^2-\alpha^2\bigr)
    \right\}.
\]
For $|\alpha|\le \rho_0$ and $n\neq0$,
\[
    (\alpha-2\pi n)^2-\alpha^2
    =4\pi^2n^2-4\pi n\alpha
    \ge 4\pi(\pi-\rho_0)|n|.
\]
Therefore, for every $0\le k\le 6$, there exist constants $C_k,c_{\rho_0}>0$, independent of $\lambda$, such that
\begin{equ}\label{eq:villain-r-bound}
    \sup_{|\alpha|\le\rho_0}
    \left|\partial_\alpha^k r_\lambda(\alpha)\right|
    \le
    C_k\lambda^k e^{-c_{\rho_0}\lambda}.
\end{equ}
Indeed, differentiating the summands only produces polynomial factors in $\lambda$ and $n$, which are absorbed by the exponential decay.

Choose $\lambda_0$ large enough that $\|r_\lambda\|_{\CC^0([-\rho_0,\rho_0])}\le 1/2$ for all $\lambda\ge\lambda_0$, and define
$  \ell_\lambda(\alpha)\eqdef \log(1+r_\lambda(\alpha))$
so that 
\begin{equ}[e:h-and-ell]
    h_\lambda(\alpha)
    =
    \frac{\alpha^2}{2c_\lambda}
    -
    \frac{1}{c_\lambda\lambda}\ell_\lambda(\alpha).
\end{equ}
By \eqref{eq:villain-r-bound} and the elementary fact that 
every $k$-th derivative of
$\ell_\lambda$ is a finite sum of terms of the form
$ \frac{\prod_{i=1}^m \partial_\alpha^{j_i}r_\lambda}{(1+r_\lambda)^m}$ with $\sum_i j_i=k$,
we obtain, after adjusting the constant $C_k$, 
\begin{equation}\label{eq:villain-ell-bound}
    \sup_{|\alpha|\le\rho_0}
    \left|\partial_\alpha^k \ell_\lambda(\alpha)\right|
    \le
    C_k\lambda^k e^{-c_{\rho_0}\lambda},
    \qquad 0\le k\le 6,
    \quad \lambda\ge\lambda_0.
\end{equation}
Since
  $\log W_\lambda(\alpha)
    =
    -\frac{\lambda}{2}\alpha^2+\ell_\lambda(\alpha)$,
    we choose 
\[
    c_\lambda
    =
    -\frac{1}{\lambda}\partial_\alpha^2\log W_\lambda(0)
    =
    1-\frac{1}{\lambda}\ell_\lambda''(0)
\]
so that $h_\lambda''(0)=1$. 
By \eqref{eq:villain-ell-bound}, after increasing $\lambda_0$ if necessary,  for all $\lambda\ge\lambda_0$
\[
    c_\lambda=1+O(\lambda e^{-c_{\rho_0}\lambda}) \in [1/2,2]\;.
\]
Combining \eqref{e:h-and-ell}, \eqref{eq:villain-ell-bound} and the bound on $c_\lambda^{-1}$ gives
\[
    \sup_{\lambda\ge\lambda_0}
    \sup_{|\alpha|\le\rho_0}
    \max_{0\le k\le 6}
    \left|\partial_\alpha^k h_\lambda(\alpha)\right|
    <\infty.
\]
\end{proof}

\begin{remark}\label{rem:Gross-assumptions}
Comparing with   \cite{MR728862} where it was assumed that 
(1) $h:\R\to\R$ is even, $2\pi$-periodic, and has two continuous derivatives;
(2) $\min h = h(0) = 0$;
(3) $h''(x)=1$ whenever $h(x)=0$;
our Assumption~\ref{assmp:YMAction} is mainly concerned about the behavior of $h$ in a neighborhood of $0$.
\end{remark}

Upon scaling $x\mapsto \e x$ as in \eqref{e:scaling-x} and choosing
\[
    \beta=\beta_\e\eqdef \e^{d-4},
    \qquad
    \vartheta=\e\theta,
\]
as in \eqref{e:scaling-theta}-\eqref{eq:barbeta1},
we obtain from \eqref{eq:mu_Nvartheta} the measure
\[
Z_\e^{-1}e^{-2\mathcal S_\e(\theta)}\,d\theta,
\qquad
\mathcal S_\e (\theta)
=
\e^{d-4}
\sum_{p\in\oplaq_\e}
 h(\e^2\mrd_\e\theta(p)),
\]
on real-valued $1$-forms $\theta$, represented on $\obonds_\e$ by values in $(-\pi/\e,\pi/\e]$.

Before introducing the Langevin dynamics, 
let $\rho_0$ be as in Assumption~\ref{assmp:YMAction}, fix
$\rho\in(0,\rho_0)$. For instance, for Wilson, Manton and Villain models, we can take $\rho_0 = \pi/4$ and  $\rho=\pi/5$. The role of the parameter $\rho$ will be apparent in the discussion below the proof of Lemma \ref{lem:deriveSPDE}.
Set
\begin{equ}[e:def-tau]
    \tau_\eps
    \eqdef
    \inf\left\{
        t\ge0:
        \exists e\in \obonds_\e,\ 
        |\eps\theta_t(e)| \ge \frac{\rho}{4}
    \right\}.
\end{equ}

On the interval $t<\tau_\eps$, one has
\begin{equ}[e:d-theta-localized]
    |\eps^2 \mrd_\eps\theta_t(p)|
    =
    |\eps\theta_t(\partial p)|
    \le \rho
\end{equ}
for every plaquette $p \in \oplaq_\e$, so all Taylor expansions of $h$ used below are taken
inside the neighborhood from Assumption~\ref{assmp:YMAction}.

The DeTurck-gauge-fixed Langevin dynamics for the above measure are therefore driven by the $L^2$ gradient $-\nabla\mathcal S_\e$, together with the DeTurck term
$- \mrd_\e \mrd^*_\e \theta$:
\begin{equ}
	\label{eq:NParaDisc}
	\partial_t \theta^\e(e)  
	=  -\e^{d-4}\sum_{p\in \oplaq_\e} \nabla  h(\e^2\mrd_\e \theta^\e(p))(e)
	- \mrd_\e \mrd^*_\e \theta^\e(e)
	 + \xi^\e(e)
\end{equ} 
for $t<\tau_\e$ where 
$\tau_\e$ is defined in \eqref{e:def-tau}.
The DeTurck term does not change the gauge equivalence class,
since it is exact, i.e., is of the form $\mrd_\e (\cdots)$.

Our main result states that the above Langevin dynamics converge to the solution of the $1$-form-valued stochastic heat equation in the discrete-continuum H\"older distances $\CC^{-1/2-}_{\gamma,\e}$ defined in Section~\ref{sec:spaces}; see also Remark~\ref{rem:scalar-vector-convention} for our convention for norms on $1$-forms.
\begin{theorem}
\label{thm:main}
Assume that $d=3$ and $\bar\beta=1$. Let $\xi=(\xi_1,\xi_2,\xi_3)$, where $\xi_1, \xi_2,\xi_3$ are independent space-time white noises on $\R\times\T^3$, defined on a probability space $(\Omega,\mathcal{F},\mathbb{P})$. Denote by $\xi^\eps$ its lattice approximation defined in \eqref{eq:Disc_white_noise}.

Suppose $I$ and $\{h_\lambda\}_{\lambda\in I}$ satisfy Assumption~\ref{assmp:YMAction}.
For every sufficiently small $\eps$, choose $\lambda_\eps\in I$ and write $h_\eps = h_{\lambda_\eps}$. Let $\theta^\eps$ solve the DeTurck-gauge-fixed lattice Langevin equation \eqref{eq:NParaDisc}, with $h=h_\eps$ and possibly random initial condition $\theta^\eps(0)$, up to the stopping time $\tau_\eps$.

Let $\theta=(\theta_1,\theta_2,\theta_3)$ solve the continuum stochastic heat
equation
\[  \partial_t\theta_i
    =
    \Delta\theta_i+\xi_i
\]
 with possibly random initial condition 
 $ \theta_i(0)$ 
 for  $i=1,2,3$.
Fix $\kappa>0$ sufficiently small and suppose that 
\[
-\frac23+\frac53\kappa<\eta<0,
\qquad
\gamma=\bigl(\eta+\tfrac12\bigr)\wedge 0.
\]
Suppose further that there is a deterministic constant $M<\infty$ such that, almost surely,
\[
\sup_{0<\e\le1}\|\theta^\e(0)\|_{\mathcal C_\e^\eta}
+\|\theta(0)\|_{\mathcal C^\eta}
\le M.
\]
If
$\|\theta^\e(0);\theta(0)\|_{\mathcal C_\e^\eta}\to0$
in probability, then for every $T>0$ and $\delta>0$,
\begin{equ}
\lim_{\eps\to0}
\mathbb P\left(
\tau_\eps>T,
\ \|\theta^\eps;\theta\|_{\CC_{\gamma,\eps}^{-1/2-\kappa,T}}
\leq \delta
\right)=1.
\end{equ}
\end{theorem}

The proof of Theorem~\ref{thm:main} will be given at the end of Section~\ref{sec:solution}.

\section{Derivation of the discrete SPDE}
\label{sec:derivation}

We now identify the leading order polynomial nonlinear term \eqref{eq:NParaDisc}.
Below, for an odd function $R$, we interpret $R(\mrd_\e \theta)$
 plaquette-wise, namely, for each plaquette $p$,
\[
(R(\mrd_\e \theta)) (p) = R(\mrd_\e \theta (p)),\qquad \mbox{e.g.} \quad 
((\mrd_\e \theta)^3)(p) = (\mrd_\e \theta(p))^3\;.
\]
These are well-defined discrete $2$-forms, satisfying the requirement  $F(\overleftarrow{p})=-F(p)$ as in 
Section~\ref{sec:Notation}.
In particular, one can apply $\mrd_\e^*$ on them.

As mentioned above, we will drop the dependence of $h$ on $\lambda_\eps$ in our notation.

\begin{lemma}\label{lem:deriveSPDE}
Let $\tau_\e$ be a stopping time defined as in \eqref{e:def-tau}. On the time interval $t<\tau_\e$,  the DeTurck-gauge-fixed Langevin dynamics \eqref{eq:NParaDisc} can be written as
\begin{equ}
\label{eq:NParaDisc2}
\partial_t \theta^\e(e)  
 = \Delta_\e \theta^\e
+ a_3 \, \e^4 \mrd_\e^* \,(\mrd_\e \theta^\e)^3(e)
+ \e^{-2} \mrd_\e^* (\mathfrak{R}(\e^2 \mrd_\e \theta^\e))(e)
+\xi^\e
\end{equ}
where  $a_3 = -\frac16 h^{(4)}(0)$ and $\mathfrak{R}(z)\eqdef
    z-a_{3}z^3-h'(z)$.

Moreover, there exists $C>0$, independent of $\eps$, such that, for all $z\in [-\rho_0,\rho_0]$,
\begin{equ}[e:coeff-bounds]
|a_3|\le C,
\qquad
|\mathfrak R(z)|\le C|z|^5.
\end{equ}
\end{lemma}

\begin{example}
For the Wilson model with $h(x) = 1-\cos x$ as in \eqref{e:Wilson-h}, Taylor expanding $h'(x)$ gives
\[
a_3 = \frac{1}{6},
\qquad
\mathfrak{R}(z)=z-\frac1{6} z^3-\sin(z)\;.
\]
For the  Manton action, $h(z)=z^2/2$ on $|z|<\pi$.
Therefore, on the localized time interval, $h'(z)=z$,
$a_3=0$ and $\mathfrak R\equiv 0$.
The equation is exactly the linear discrete stochastic heat equation.
\end{example}	
\begin{proof}[of Lemma~\ref{lem:deriveSPDE}]
The directional derivative of $\mathcal S_\e$ with respect to an arbitrary $1$-form $v$ is given by
\begin{equs}
\frac{d}{ds}\bigg|_{s=0} 
\mathcal S_\e (\theta^\e + s\, v) 
&= \eps^{d-4} \sum_{p \in \oplaq_\eps} h'(\eps^2 \mrd_\e \theta^\e(p)) \,\eps^2 \mrd_\e v(p) \\
&= \eps^{-2} \langle h'(\eps^2 \mrd_\e \theta^\e), \mrd_\e v \rangle_\e^{(2)}
= \eps^{-2} \langle \mrd_\e^* h'(\eps^2 \mrd_\e \theta^\e), v \rangle_\e^{(1)}
\end{equs}
where we used \eqref{e:adjoint-d} in the last step.
Therefore, \eqref{eq:NParaDisc} can be written as
\[
\partial_t \theta^\e =- \eps^{-2}  \mrd_\e^* h'(\eps^2 \mrd_\e \theta^\e) 
- \mrd_\e \mrd^*_\e \theta^\e
	 + \xi^\e \;.
\]
Recall from \eqref{e:d-theta-localized} that before $\tau_\eps$,
$|\eps^2 \mrd_\e \theta^\e|$ is bounded by $\rho$,
and $h\in \CC^6$ in $(-\rho,\rho)$. 
This allows us to Taylor expand $h'$:
\begin{equs}
h'(\eps^2 \mrd_\e \theta^\e)  
= h'(0) &+ h''(0) \eps^2 \mrd_\e \theta^\e + \frac12 h'''(0) (\eps^2 \mrd_\e \theta^\e)^2
\\
& + \frac16 h^{(4)}(0) (\eps^2 \mrd_\e \theta^\e)^3 + \frac{1}{24} h^{(5)}(0) (\eps^2 \mrd_\e \theta^\e)^4
- \mathfrak{R}(\eps^2 \mrd_\e \theta^\e)\;.
\end{equs}
By Assumption~\ref{assmp:YMAction},
$h''(0)=1$, and 
 $h'(0)=h'''(0)=h^{(5)}(0) =0$.
Observe that for the linear term, following \eqref{e:def-Laplace}, we have
\[
-\eps^{-2} \mrd_\e^* (\eps^2 \mrd_\e \theta^\e )- \mrd_\e \mrd^*_\e \theta^\e = \Delta_\e \theta^\e\;,
\]
which yields \eqref{eq:NParaDisc2}.

By Taylor's theorem with integral remainder,
\[
    \mathfrak R(z)
    =
    -\frac1{24}
    \int_0^z
        (z-y)^4 h^{(6)}(y)\,dy \,.
\]
The uniform $\CC^6$ bound in Assumption~\ref{assmp:YMAction} gives
\[
    |\mathfrak R(z)|
    \le
    C
    \int_0^{|z|}
        (|z|-y)^4\,dy
    \le
    C|z|^5 .
\]
\end{proof}

Note that, even to derive the equation \eqref{eq:NParaDisc2}
we need to assume $t<\tau_\eps$. 
Eventually, we will prove that over an arbitrary {\it fixed} interval $[0,T]$ the equation \eqref{eq:NParaDisc2} converges to SHE. To this end we need to argue that $\tau_\eps>T$ with high probability;
this argument will rely on PDE estimates in Section~\ref{sec:solution}, in particular, a decomposition of $\theta$ into the linear and remainder part.
However, these arguments depend on the fact that 
\eqref{eq:NParaDisc2} makes sense in the first place,
which is only known before $\tau_\eps$ as in Lemma~\ref{lem:deriveSPDE}. To avoid a circular argument, recall $\rho\in (0,\rho_0)$ introduced  above \eqref{e:def-tau}, define 
\footnote{For instance, the interpolation can be chosen as $H(z)=z+\chi(z)(h'(z)-z)$ 
where $\chi\in \CC^\infty_c([-\rho_0,\rho_0])$ is even with $\chi(z)=1$ for $z\in[-\rho,\rho]$.}
\[
H(z)=\begin{cases*}
h'(z), \qquad  |z|\le \rho,
\\
\mbox{arbitrary $\CC^5$ interpolation},
\\
z , \qquad |z| \mbox{ sufficiently large},
\end{cases*}
\]
which is odd, satisfying $H(-z)=-H(z)$.
We then define $\widehat{\mathfrak{R}}(z)\eqdef  z-a_{3}z^3-H(z)$ which coincides with $\mathfrak{R}$ on $[-\rho,\rho]$,
and satisfies
$ |\widehat{\mathfrak{R}}(z)|\le C |z|^5 $ globally, i.e., for all $z\in \R$ and $\eps$. 
We study
\begin{equ}
\label{eq:NParaDisc2'}
\partial_t \widehat\theta^\e(e)  
 = \Delta_\e \widehat\theta^\e
+ a_3 \, \e^4 \mrd_\e^* \,(\mrd_\e \widehat\theta^\e)^3(e)
+ \e^{-2} \mrd_\e^* (\widehat{\mathfrak{R}}(\e^2 \mrd_\e \widehat\theta^\e))(e)
+\xi^\e
\end{equ}
with initial condition $\widehat\theta^\e(0)=\theta^\e(0)$. This equation is well-defined
on any interval $[0,T]$,
whereas \eqref{eq:NParaDisc2} only makes sense on $[0,\tau_\eps]$.
Since  $\widehat{\mathfrak{R}}$ is odd, 
$\widehat{\mathfrak{R}}(\e^2 \mrd_\e \widehat\theta)$ is a well-defined $2$-form, 
so $ \mrd_\e^* $ can be applied to it as discussed before 
Lemma~\ref{lem:deriveSPDE}.
Moreover, for any {\it fixed} $\e>0$ , \eqref{eq:NParaDisc2'} is well-posed on $[0,T]$. Indeed,
it is a stochastic ODE for any {\it fixed} $\e>0$.
By definition, 
\[
a_3 z^3 
+\widehat{\mathfrak{R}}(z) 
=
z-H(z).
\]
The assumption on $H$ implies that $z \mapsto a_3 z^3 
+\widehat{\mathfrak{R}}(z)  $ is 
compactly supported and globally Lipschitz.
Therefore, well-posedness on $[0,T]$ follows from standard stochastic ODE theory.

We will show that 
$\sup_{t\le T}  \eps \|\widehat\theta^\e(t)\|_{L^\infty}<\rho/4 $
on a good event $\Omega_{r,\eps}(T)$, 
defined in \eqref{eq:GoodEvent}. This then implies that over $[0,T]$, for every plaquette $p$, 
one has $ |\eps^2 \mrd_\eps\widehat\theta^\e(p)|\le 4\eps \|\widehat\theta^\e\|_{L^\infty}  < \rho$
and therefore $H=h'$; namely \eqref{eq:NParaDisc2'} and \eqref{eq:NParaDisc2} are identical.

\begin{remark} 
  \label{Rmk:Scaling_vague}
Informally speaking, we have derived a discrete SPDE of the following form:
\begin{equ}[e:vague-form]
\partial_t \theta =\Delta \theta 
+ \eps^4 \partial^4 \theta^3 
+ \eps^8 \partial^6 \theta^5 + \cdots + \xi^\eps .
\end{equ}
By power counting as in \eqref{e:scaling-meas}, 
we have $\Delta\theta\sim -(d+2)/2$ and
$\eps^4 \partial^4 \theta^3  \sim -\frac32 (d-2)$, etc.
The ``sub-criticality'' condition $-(d+2)/2 <  -\frac32 (d-2)$ 
is therefore again equivalent to$d<4$. When $d=3$, we have the following (formal) scalings 
\[
\Delta \theta \sim -5/2, \quad 
\eps^4 \partial^4 \theta^3 \sim -3/2,\quad
\eps^8 \partial^6 \theta^5 \sim -1/2, \quad \mbox{etc.}
\]
In fact, this is precisely the same scaling as dynamical $\Phi^4_3$.
\end{remark}

\begin{remark}
The above discrete equation is considerably simpler than the Lie algebra valued equation obtained for non-abelian lattice gauge theories
in \cite[Section~3]{chevyrev2023invariant} (note that the discrete equation in \cite[Proposition~3.13]{chevyrev2023invariant} also holds in 3D). For $U(1)$, the above simple derivation is effectively equivalent to the geometric derivation in \cite[Section~3]{chevyrev2023invariant};
 the Baker--Campbell--Hausdorff formula therein reduces to the trivial identity $e^a e^b = e^{a+b}$, and the
differential of the exponential map is the identity.
Thus, in the abelian case, pulling the group-valued Langevin dynamics back to the Lie algebra produces exactly an additive white noise; whereas in the non-abelian
case, the same pullback yields field-dependent coefficients multiplying the noise, as well as additional commutator terms.

Moreover, in the abelian case considered here, the only nonlinearities in the equation
arise from the non-quadratic part of the single-plaquette
action. They are suppressed by explicit positive powers of $\eps$. This contrasts with the non-abelian continuum
equation, where commutator nonlinearities of the form $A\partial A$ and $A^3$ remain at order one. Below we show that in the present abelian setting all nonlinear terms vanish as $\eps\to 0$.
\end{remark}

\section{Discrete function spaces}
\label{sec:spaces}

Let $O_T = [0,T]\times \T^d$ and $O_T^\eps = [0,T]\times \T^d_\eps$, and write $\bar O_T = [-1,T+1]\times \T^d$ and $\bar O_T^\eps = [-1,T+1]\times \T^d_\eps$. Let $\| \; \|_\s$ denote the parabolic distance on $\R\times\T^d$.
We follow the definitions of discrete H\"older spaces $\CC_\e^\alpha$ and $\CC_\e^{\alpha,T}$ as in
\cite[Section~4.1]{chevyrev2023invariant}, but allowing $\alpha\in \R$.

Let $\alpha > 0$.
For functions $f\in\CC_\s^\alpha (O_T)$ 
 and $f^\e\colon  O_T^\eps \to \R$, we define
\begin{equation}\label{e:discHolderT}
    \begin{aligned}
\|f^\e;f\|_{\CC_\e^{\alpha,T}}
&\eqdef 
\sup_{z\in O_T^\eps}\sum_{|k|=0}^{\lfloor\alpha\rfloor}|D^k f(z)-D^{k}_\e f^\e(z)|
+ \sup_{\substack{z, w\in O_T \\ \|z-w\|_\s < \e}}
\sum_{|k|=\lfloor\alpha\rfloor}
\tfrac{|D^{k}f(z)-D^{k}f(w)|}{\|z-w\|_\s^{\alpha-\lfloor\alpha\rfloor}}
\\
& + \sup_{\substack{z, w\in O_T^\eps \\ \|z-w\|_\s \geq \e}} 
\sum_{|k|=\lfloor\alpha\rfloor}
	\frac{|(D^{k}f(z)-D^{k}f(w))-(D^{k}_\e f^\e(z)-D^{k}_\e f^\e(w))|}{\|z-w\|_\s^{\alpha-\lfloor\alpha\rfloor}}\;.
    \end{aligned}
\end{equation}
Above, we write $\lfloor\alpha\rfloor$ for the floor of $\alpha$.
The sums in \eqref{e:discHolderT} are over all multi-indices $k=(k_0,\cdots,k_d)$ with $|k|\eqdef 2k_0+k_1+\cdots +k_d$ of the specified value, and we write
$D^k= \partial_t^{k_0} \partial_{x_1}^{k_1}\cdots \partial_{x_d}^{k_d} $.
Furthermore, for $j\in[d]$, let
\[
D_{\e,j}f(x)=\e^{-1}\bigl(f(x+\e_j)-f(x)\bigr)
\]
and define $D_\eps^k=\partial_t^{k_0} D_{\e,1}^{k_1}\cdots D_{\e,d}^{k_d} $. 


Now let $\alpha\leq 0$.
For discrete distributions $f^\e$ on $\bar O_T^\eps$
and distributions $f$ on $\bar O_T$ 
we define 
\begin{equ}\label{eq:f_e_f_alpha_neg}
\Vert f^\e ;f \Vert_{\CC^{\alpha,T}_\e} 
\eqdef 
\sup_{\varphi \in \CB^r_0} 
\sup_{z \in O_T^\eps} 
\sup_{\lambda \in [\e,1]} 
\lambda^{-\alpha} 
|\langle f^\e, \varphi_z^\lambda \rangle_\e
-\langle f, \varphi_z^\lambda \rangle | \;,
\end{equ}
where $r$ is the smallest integer such that $r>-\alpha$,
 and $\CB^r_0$ is the set of space-time smooth functions $\phi\colon \R\times\R^d \to\R$ with $\|\phi\|_{\CC^r}\leq 1$
supported in a ball of radius $\frac14$ centered at $0$, with $\phi^\lambda_z$ being its parabolic rescaling around $z$,
and $\langle \; , \; \rangle_\e$ is the discrete approximation of the inner product 
$\langle \; , \; \rangle$. One can find more details in \cite[Section~4.1]{chevyrev2023invariant}.

We then 
define $\Vert f^\e \Vert_{\CC^{\alpha,T}_\e} =\Vert f^\e;0 \Vert_{\CC^{\alpha,T}_\e} $
where the right-hand side is as in 
\eqref{e:discHolderT} or \eqref{eq:f_e_f_alpha_neg}.

Again as in \cite[Section~4.1]{chevyrev2023invariant}, for $\alpha\in \R$, discrete functions $f^\eps$ on $\T_\e^d$ and distributions $f$ on $\T^d$,
we define 
\begin{align}
\label{def:Holder_disc_cts}
\|f^\e ; f\|_{\CC^\alpha_\e}\;,
\qquad
\|f^\e \|_{\CC^\alpha_\e} \eqdef \|f^\e ; 0\|_{\CC^\alpha_\e}
\end{align}
as in 
\eqref{e:discHolderT} or \eqref{eq:f_e_f_alpha_neg}, but with $O_T$, $O_T^\e$ replaced by 
$\T^d$, $\T^d_\e$,  space-time derivatives or finite differences replaced by the spatial derivatives or finite differences, and space-time test functions $\phi$ replaced by the spatial ones.
We frequently use the suggestive notation $\CC^0_\e = L^\infty_\e$ since $\|f\|_{\CC^0_
\e}$ is equivalent to $\sup_{x\in\T^d_\e} |f(x)|$ uniformly in $\e$.

We will also need ``inhomogeneous'' norms which distinguish time and space and allow possible blow-up at $t=0$.
For $f^\e\colon O_T^\e \to \R$, $f \colon O_T \to \R$, the space-time Hölder norms with parameters $\alpha \in \R$,  $\eta \le 0$ are defined as
\begin{equ}
  \| f^\e;f\|_{\CC^{\alpha,T}_{\eta, \e}} \eqdef \sup_{t \in (0, T]}
(t^{\frac12} \wedge 1)^{-\eta} \|f^\e(t) ;f (t)\|_{\CC^\alpha_\e} \;,
\qquad
    \| f^\e\|_{\CC^{\alpha,T}_{\eta, \e}} \eqdef \| f^\e;0\|_{\CC^{\alpha,T}_{\eta, \e}} 
    \;.
\end{equ}



The following lemma will be useful. 
When the direction is not important, we write $D_\e$ for any one of the operators $D_{\e,j}$; all estimates below are uniform in $j$.
\begin{lemma}
\label{lem:eps-improve-reg}
There exists $C>0$ independent of $\eps$ such that the following hold.
\begin{enumerate}[label=(\roman*)]
\item \label{pt:wo_D} 
For all $\alpha\in \R$ and $\beta \geq 0$, one has
 $\norm{ f}_{\CC^{\alpha+\beta}_\eps} \leq C \e^{-\beta} \norm{f}_{\CC^{\alpha}_\eps}$.
 \item \label{pt:w_D}
 For all $\alpha \in \R$, one has $\norm{ D_\e f}_{\CC^{\alpha-1}_\eps} \leq C  \norm{f}_{\CC^{\alpha}_\eps}$.
  \item \label{pt:w_Deps}
 For all  $\alpha \leq \bar\alpha+1$, one has $\norm{\e D_\e f}_{\CC^{\bar \alpha}_\eps} \le C \e^{\alpha-\bar\alpha} \norm{ f}_{\CC^{ \alpha}_\eps} $. 
\end{enumerate}
\end{lemma}


\begin{proof}

 \ref{pt:wo_D}. When $\alpha \leq 0$ and $\alpha+\beta\leq 0$, the bound follows from (the spatial version of) definition \eqref{eq:f_e_f_alpha_neg}: 
 \begin{equs}
\Vert f^\e  \Vert_{\CC^{\alpha+\beta}_\e} 
&= 
\sup_{\varphi ,z} 
\sup_{\lambda \in [\e,1]} 
\lambda^{-\alpha-\beta} 
|\langle f^\e, \varphi_z^\lambda \rangle_\e | 
\\
&
\le
\eps^{-\beta} 
\sup_{\varphi ,z} 
\sup_{\lambda \in [\e,1]} 
\lambda^{-\alpha} 
|\langle f^\e, \varphi_z^\lambda \rangle_\e | 
=
\e^{-\beta} \norm{f}_{\CC^{\alpha}_\eps}\;.
\end{equs}

When $\alpha\leq 0$, $\alpha+\beta>0$, we proceed via the  $L^\infty_\eps$ norm
\begin{equ}[e:passing-L-inf]
\norm{ f}_{\CC^{\alpha+\beta}_\eps} \leq C\e^{-\alpha-\beta} \norm{ f}_{L^\infty_\eps} \leq \tilde{C}\e^{-\beta} \norm{ f}_{\CC^{\alpha}_\eps}.
\end{equ}
Indeed, for the first inequality, we need to bound the  $D_\e^k f$ term and the H\"older term 
on the right-hand side of \eqref{e:discHolderT}.
By definition of the finite difference one has $\|D_\e f\|_{L^\infty_\eps} \le 2 \eps^{-1} \|f\|_{L^\infty_\eps} $
and therefore
\[
\|D_\e^k f\|_{L^\infty_\eps} 
\le C \eps^{-|k|} \|f\|_{L^\infty_\eps} 
\le C \eps^{-\alpha-\beta} \|f\|_{L^\infty_\eps} 
\]
 for every $k$ such that $0 \le |k|  \le \lfloor\alpha+\beta \rfloor$.
For the H\"older term, writing $k=\lfloor\alpha+\beta\rfloor$,
\[
\sup_{ |x-y|\geq \e}
	\frac{|D^k_\e f^\e(z)-D^k_\e f^\e(w)|}{|x-y|^{\alpha+\beta-k}}
\le 
2 \eps^{-(\alpha+\beta-k)}\|D_\e^k f\|_{L^\infty_\eps} 
\le
C \eps^{-(\alpha+\beta)} \|f\|_{L^\infty_\eps}
\]
where we used the bound on $\|D_\e^k f\|_{L^\infty_\eps} $ we just proved.

For the second inequality in \eqref{e:passing-L-inf},
recall that  for $\alpha \leq 0$
\[
    \|f\|_{\mathcal{C}_\e^\alpha} 
    = \sup_{\varphi ,x} \sup_{\lambda \in [\e, 1]} \lambda^{-\alpha} |\langle f, \varphi_x^\lambda \rangle_\e|
    \ge 
    \sup_x
    \eps^{-\alpha} |\langle f, \psi_x^\eps \rangle_\eps| 
    \gtrsim 
    \sup_x
    \eps^{-\alpha} |f(x)| \;.
\]
Here, the first inequality follows from choosing $\lambda=\eps$ and any
$\psi \in \mathcal{B}_0^r$  such that $\psi(0) > 0$. 
Since the support size of $\psi$ is less than $\frac14$, we know that $\psi^\eps_x$
(which is $\psi$ rescaled by $\eps$ around $x$)
 is equal to zero on all the lattice sites of $\T^d_\eps$
except for $x$ at which $\psi^\eps_x \gtrsim \eps^{-d}$, so the last step above follows from the definition
\[
    \langle f, \psi_x^\eps \rangle_\eps 
    = \eps^d \sum_{y \in \T_\eps^d} f(y)
    \psi^\eps_x(y) \;.
\]

Finally, when  $\alpha > 0$, which implies $\alpha+\beta> 0$, the proof follows similarly to the first inequality of \eqref{e:passing-L-inf}.

\ref{pt:w_D}. 
For $\alpha\geq1$, the claim follows from the definition of the corresponding Hölder norm. For $\alpha< 1$, recall the definition of the discrete negative H\"older norm
\[ 
\| g\|_{\Ceps{\beta}} = \sup_{\varphi \in \mathcal{B}_0^r} \sup_{x \in \T_\eps^d} \sup_{\lambda \in [\eps, 1]} \lambda^{-\beta} \big| \inner{g}{\rescaled{\varphi}{x}{\lambda}} \big|, \quad \rescaled{\varphi}{x}{\lambda}(y) = \lambda^{-d} \varphi\left(\tfrac{y-x}{\lambda}\right)\;.
\]
Fix $j\in[d]$. To bound $\|D_{\e,j}f\|_{\Ceps{\alpha-1}}$, consider $\lambda \in [\eps,1]$, $h\in (0,1]$, and let $D_{-\lambda h,j} g = \frac{g(\cdot) - g(\cdot - \lambda he_j)}{\lambda h}$ denote the backward difference in direction $e_j$ with step length $\lambda h$, and define
\begin{align}
\label{eq:Psi_newtestF}
\Psi^{h,j}(u)
&\eqdef \frac{\varphi(u)-\varphi(u-he_j)}{h},
\\
D_{-\lambda h,j}(\rescaled{\varphi}{x}{\lambda})(y)
&=\lambda^{-1}\rescaled{(\Psi^{h,j})}{x}{\lambda}(y)\;.
\end{align}
The family $\Psi^{h,j}$ is uniformly admissible as a family of test functions, since
$\|D^k\Psi^{h,j}\|_{L^\infty}\lesssim\|\varphi\|_{\CC^{k+1}}$
and the support of $\Psi^{h,j}$ is contained in a ball of radius $O(1)$
uniformly in $h\in(0,1]$, $j\in[d]$, and $\varphi \in \mathcal{B}^{k+1}_0$. Taking $h=\e/\lambda$ and using summation by parts yields
\begin{equation}
\label{eq:Inner_w_Dtest}
\left|\inner{D_{\e,j} f}{\rescaled{\varphi}{x}{\lambda}}\right|
=\left|\inner{f}{D_{-\e,j}(\rescaled{\varphi}{x}{\lambda})}\right|
=\lambda^{-1}\left|\inner{f}{\rescaled{(\Psi^{h,j})}{x}{\lambda}}\right|\;.
\end{equation}
When $\alpha \leq 0$, we obtain directly from the definition of $\Ceps{\alpha}$ that 
\eqref{eq:Inner_w_Dtest} $\lesssim 
\| f\|_{\Ceps{\alpha}} \lambda^{\alpha-1}$.
(For this, remark that the norm on $\Ceps{\alpha}$ remains equivalent, uniformly in $\e$, if the fixed
support radius $1/4$ in the definition of $\CB^r_0$ is replaced by
any other fixed positive radius.)

When $0<\alpha<1$, we subtract $f(x)$ from $f$, which does not change $D_{\e,j}f$, and obtain
\begin{align*}
\lambda^{-1}\left| \inner{f}{\rescaled{(\Psi^{h,j})}{x}{\lambda}}\right|  \lesssim  \lambda^{-1}\left| \inner{\norm{f}_{\Ceps{\alpha}}(\cdot -x)^\alpha}{\rescaled{(\Psi^{h,j})}{x}{\lambda}}\right| \leq\lambda^{-1} \left| \inner{\norm{f}_{\Ceps{\alpha}}\lambda^\alpha}{\rescaled{(\Psi^{h,j})}{x}{\lambda}}\right|.
\end{align*} 
Multiplying both sides of \eqref{eq:Inner_w_Dtest} by $\lambda^{-(\alpha-1)}$ and taking the supremum over $\varphi, x, \lambda$ gives the desired bound.

\ref{pt:w_Deps}. This follows from a combination of \ref{pt:wo_D} and \ref{pt:w_D}:
\begin{align*}
\norm{\e D_\e f}_{\Ceps{\bar\alpha}} 
\overset{\text{\ref{pt:w_D}}}{\lesssim} 
\norm{\e f}_{\Ceps{\bar\alpha+1}} 
\overset{\text{\ref{pt:wo_D}}}{\lesssim}
\e \e^{\alpha-(\bar\alpha+1)} \norm{ f}_{\Ceps{\alpha}} 
 =  \e^{\alpha-\bar\alpha} \norm{f}_{\Ceps{\alpha}}.
\end{align*}
Here, $\alpha \leq \bar\alpha+1$ is used to verify the condition $\beta = \bar\alpha+1-\alpha\geq 0 $ in \ref{pt:wo_D}.
\end{proof}

\begin{remark}\label{rem:scalar-vector-convention}
For the rest of the paper,  we often use the above spaces and apply Lemma~\ref{lem:eps-improve-reg}
not only to discrete functions on $ \T^d_\e$, but also 
to a {\it fixed} component of a discrete $1$-form or $2$-form.
Since $\obonds_{\e,j}$ for a fixed $j$ can be identified with a lattice $\T^d_\e$ (by viewing  midpoints of the bonds as lattice sites), Lemma~\ref{lem:eps-improve-reg} holds for the $j$-th component of a discrete $1$-form which is a function on $\obonds_{\e,j}$. Likewise for $2$-forms.

Moreover, let us write $\theta_j$ for the $j$-th component of a $1$-form $\theta$,
and $F_{ij}$ with $i<j$ for  the $(i,j)$-th component of a $2$-form $F$
(so $F_{ij}$ is a function on the plaquettes of the form $p=(x,\e_i,\e_j)$ with $1\le i < j \le d$ fixed). We extend the definitions of the above spaces and distances to forms, for instance, for $1$-forms $\theta^\eps$ and $\theta$ we define
\begin{equ}[e:norm-forms]
    \|\theta^\eps;\theta\|_{\mathcal C_{\gamma,\eps}^{\eta,T}}
    \eqdef
    \max_{i \in [d]}
    \|\theta^\eps_i;\theta_i\|_{\mathcal C_{\gamma,\eps}^{\eta,T}},
\end{equ}
and similarly for $2$-forms.
Furthermore, 
Lemma~\ref{lem:eps-improve-reg} also holds with $D_\e$ replaced by $\mrd_\e$ and $\mrd_\e^*$
in the following sense. 
By definition  of $\mrd_\e$ and Lemma~\ref{lem:eps-improve-reg}\ref{pt:w_D},
\begin{equs}
\| (\mrd_\e \theta)_{ij}\|_{\CC_\e^{\alpha}}
&=
\| D_{\e,i} (\theta_j)- D_{\e,j} (\theta_i) \|_{\CC_\e^{\alpha}}
\le
\| D_{\e,i} (\theta_j)\|_{\CC_\e^{\alpha}} +  \| D_{\e,j} (\theta_i) \|_{\CC_\e^{\alpha}}
\\
&\lesssim
\| \theta_i \|_{\CC_\e^{\alpha+1}} +  \| \theta_j \|_{\CC_\e^{\alpha+1}}.
\end{equs}
So we have $\| \mrd_\e \theta\|_{\CC_\e^{\alpha}} \lesssim \| \theta \|_{\CC_\e^{\alpha+1}} $.
Likewise for the differentiation operator $\mrd_\e^*$.
\end{remark}

\section{Solution theory}
\label{sec:solution}

In this section, we follow the setting of Theorem~\ref{thm:main}.
 For $i\in\{1,2,3\}$, let $\Psi^\e_i$ be the solution to the following stochastic heat equation on $\obonds_i$
\begin{align}\label{eq:SHE}
\partial_t \Psi^\e_i  
 = \Delta_\e \Psi^\e_i +\xi_i^\e\;, \qquad \Psi^\e_i(0)=\theta_i^\e(0) 
\end{align}
where $\theta^\e(0)$ is  the initial condition as in Theorem~\ref{thm:main}.

Recall that in Section~\ref{sec:derivation}
we plan to analyze $\widehat\theta^\e$ with equation \eqref{eq:NParaDisc2'},
for which we already know well-posedness on $[0,T]$
for fixed $\eps>0$,
instead of \eqref{eq:NParaDisc2}.
Our strategy will be to first show the desired convergence of $\widehat\theta^\e$ to $\theta$ over any interval $[0,T]$, and then prove that $\widehat\theta^\e = \theta^\e$ on $[0,T]$ with high probability.

To this end, we decompose $\widehat\theta^\e(t) = \Psi^\e(t) + \widehat v^\e(t)$
(these are all $1$-forms,
 and we write their components as $\widehat\theta^\e(t) = (\widehat\theta_1^\e(t), \widehat\theta_2^\e(t), \widehat\theta_3^\e(t))$, etc.). 
By \eqref{eq:SHE} and \eqref{eq:NParaDisc2'},
$\widehat v^\e$ solves the following equation with $0$ initial condition:
\begin{equ}
\label{eq:v_i}
\partial_t \widehat v^\e
 = \Delta_\e \widehat v^\e
+ a_3 \, \e^4 \mrd_\e^* \,(\mrd_\e \Psi^\e+\mrd_\e \widehat v^\e)^3  +  \e^{-2} \mrd_\e^* (\widehat{\mathfrak{R}}(\e^2 (\mrd_\e \Psi^\e + \mrd_\e \widehat v^\e)))\;.
\end{equ}

As discussed in Remark \ref{Rmk:Scaling_vague}, the nonlinear terms in our equation 
have the same scaling as  in $\Phi^4_3$.
Recall that for the $\Phi^4_3$ equation, the Da Prato--Debussche decomposition is not sufficient: the products appearing in the analogous nonlinearities to those in \eqref{eq:v_i}  are not well-defined in the classical sense, and the standard ``Young--Schauder'' argument does not close.
Therefore, one typically needs additional machinery, such as regularity structures \cite{Hairer14} or paracontrolled calculus \cite{GIP15,CC18}, to give meaning to the $\Phi^4_3$ equation (though see also \cite{JP23_Phi43}), and the equation needs renormalization to obtain a convergent limiting object.
However, in our case,
these terms involve explicit powers of $\eps$. This allows us to invoke Lemma~\ref{lem:eps-improve-reg} to gain sufficient regularity,
and as in \eqref{eq:bd_Psi_i} and \eqref{e:bound-Cpq} below, we can avoid Young's inequality and 
bound the products in $L^\infty$.
Although Wick renormalization is still needed, the renormalization constants are offset by the positive powers of $\e$ (see Lemma~\ref{lm:renom_cst}).

\subsection{Deterministic estimates for  the nonlinear terms}

For the rest of this section, 
recall the convention \eqref{e:norm-forms} from Remark~\ref{rem:scalar-vector-convention}:
 the norm of $k$-form is understood as 
the maximum of the norms of its components.

Below $T>0$ is fixed, and $\kappa>0$ will be chosen small enough.
\begin{assumption}
  \label{asmpt:v}
 (1) There exists $M>0$ independent of $\e$   such that
 \[
\| \theta^\e(0)\|_{\CC^{\eta}_\e} \le M, 
\qquad  \mbox{for some} \quad  -\frac{2}{3}+\frac{5\kappa}{3}\leq\eta\leq0 .
 \]
 (2)  
There exists $R_v>0$ independent of $\e$  such that 
\begin{equ}[eq:v_bound]
\|\widehat v^\e\|_{\CC^{\alpha,T}_{0, \eps}}\leq R_v\;, \qquad \mbox{where} \quad  3\kappa<\alpha\leq 1.
\end{equ}
\end{assumption}



Below, we write $\tilde \Psi^\e_i$ for the solution to the stochastic heat equation \eqref{eq:SHE} with $0$ initial condition,
and \[
\Psi^\e_i(t)=\tilde \Psi^\e_i (t)+ P_t^\e \theta_i^\e(0)\;,
\]
where $ P_t^\e$ is the discrete semigroup for $\Delta_\eps$.

\begin{assumption}\label{as:models_abstract}
There exists $r>0$ independent of $\e$ 
 such that for all  $\e>0$,
\begin{align}\label{e:assump-Psi-1}
 \sup_{t\in [0,T]}\|\tilde \Psi^\e(t)\|_{\CC^{-1/2-\kappa}_\e} 
 &\leq r \;,
 \\
 \label{e:assump-Psi-2}
  \sup_{t\in [0,T]}  \|\e^4 \mrd_\e^* ( \mrd_\e \tilde \Psi^\e)^3(t)\|_{\CC^{-3/2-2\kappa}_\e}
  & \leq r \e^{\frac{\kappa}{2} }\;.
\end{align}
\end{assumption}


By  Lemma \ref{lem:eps-improve-reg}\ref{pt:w_Deps}, \eqref{eq:v_bound} implies for $t\in [0,T]$,
\begin{equ}
  \label{assmp:bd_v}
\norm{ \e^{1-\alpha} \mrd_\e \widehat v^\e(t)} _{L^\infty_\e} 
\lesssim \norm{\widehat v^\e(t)} _{\CC_\eps^\alpha} 
\le R_v,
\end{equ}
and \eqref{e:assump-Psi-1} implies for $t\in[0,T]$,
\begin{equ}
  \label{eq:bd_Psi_i}
  \| \e^{3/2+\kappa}\mrd_\e \tilde \Psi^\e(t) \|_{L^\infty_\e}
  \lesssim
 \| \tilde \Psi^\e(t)\|_{\CC^{-1/2-\kappa}_\e} \leq r.
\end{equ}

For the cubic term in \eqref{eq:v_i}, we rewrite it as 
\begin{align}
\label{eq:cubic_expansion}
&\eps^{4}\mrd_\e^* \left(\mrd_\e\tilde \Psi^\e (t) + \mrd_\e P^\e_t \theta^\e(0) + \mrd_\e \widehat v^\e(t)\right)^3\nonumber 
=
\sum_{p=0}^3 \sum_{q=0}^{3-p}
 \binom{3}{p} \binom{3-p}{q}\mathfrak{C}^{p,q} (t),
\end{align}
where, for $p,q\ge 0$ with $p+q\leq 3$,
\begin{equ} 
  \label{eq:Cpq} 
  \mathfrak{C}^{p,q}(t)\eqdef\eps^{4}\mrd_\e^* \left((\mrd_\eps \tilde \Psi^\e)^p(t) (\mrd_\e P^\e_t \theta^\e(0))^q (\mrd_\e \widehat v^\e)^{3-p-q}(t)\right).
\end{equ}

The  term $\mathfrak{C}^{3,0}$ is bounded by assumption \eqref{e:assump-Psi-2}. For the other terms, we have the following bound.
\begin{lemma}
\label{lm:cubic_bd}
Under Assumptions \ref{asmpt:v} and  \ref{as:models_abstract}, 
there exists $C>0$
such that for nonnegative integers $p, q$ satisfying $p<3$ and $ p+q\leq 3$, we have
\begin{equ}[e:cubic_bd]
\norm{\mathfrak{C}^{p,q}(t)}_{\CC^{-1}_{\e}}\leq 
C\e^{\kappa}  
r^p \,M^q\,R_v^{3-p-q} \,
(1\wedge t)^{
\frac{1+3\eta-\kappa}{2}
\wedge
\frac{\eta-3\kappa}{2} \wedge 0 
}.
\end{equ}
\end{lemma}
\begin{proof}

By Lemma \ref{lem:eps-improve-reg}\ref{pt:w_Deps}, for $p<3$,
\begin{equs}[e:bound-Cpq]
{}&\norm{\mathfrak{C}^{p,q}(t)}_{\CC^{-1}_{\e}} 
= \norm{\eps^{4}\mrd_\e^* \left((\mrd_\eps \tilde \Psi^\e)^p (t)(\mrd_\e P^\e_t \theta^\e(0))^q (\mrd_\e \widehat v^\e)^{3-p-q}(t)\right)}_{\CC^{-1}_{\e}}
\\
&\quad \lesssim \norm{\eps^{4} (\mrd_\eps \tilde \Psi^\e)^p (t)(\mrd_\e P^\e_t \theta^\e(0))^q (\mrd_\e \widehat v^\e)^{3-p-q}(t)}_{L^\infty_{\e}}
\\
&\quad \leq \e^{\kappa} \|\e^{3/2+\kappa} \mrd_\e \tilde \Psi^\e\|^p_{L^\infty_\e}  \| \e^{1+\nu} \mrd_\e P^\e_t \theta^\e(0) \|_{L^\infty}^q  \| \e^{1-\alpha} \mrd_\e \widehat v^\e\| _{L^\infty_\e}^{3-p-q},
\end{equs}
where
 for $q\neq0$,
 we introduce $\nu$ such that
\[
\eps^4 = \eps^\kappa \left(\eps^{3/2+\kappa}\right)^p \left(\eps^{1+\nu}\right)^q \left(\eps^{1-\alpha}\right)^{3-p-q}.
\]
Note that for all $p, q$ satisfying $p<3$ and $ p+q\leq 3$, we have
\begin{equ}[e:cubic-4]
\kappa+ p(3/2+\kappa) +(3-p-q)(1-\alpha) <4
\end{equ}
because the maximum of the left-hand side occurs at $(p,q)=(2,0)$ in which case the left-hand side 
is $4+3\kappa-\alpha<4$ using the assumption $ 3\kappa<\alpha$.
This justifies the last inequality in \eqref{e:bound-Cpq} and ensures $\nu\ge -1$ when $q>0$.

Notice that for $t\in[0,T]$, following Lemma \ref{lem:eps-improve-reg}\ref{pt:w_Deps} (the condition of which is verified since $\nu\ge -1$)
and the heat kernel estimate, we have 
\begin{align}
\label{eq:bd_initial}
\norm{\e^{1+\nu} \mrd_\e P^\e_t \theta^\e(0)}_{L^\infty_\e} \leq \norm{P^\e_t \theta^\e(0)}_{\CC^{-\nu}_\e} \leq  (1\wedge t)^{\frac{\eta+\nu}{2}\wedge 0 }\norm{\theta^\e(0)}_{\CC^\eta_\e}.
\end{align} 
Applying  \eqref{assmp:bd_v} for $\norm{ \e^{1-\alpha} \mrd_\e \widehat v^\e(t)} _{L^\infty_\e}$, \eqref {eq:bd_Psi_i} for $\| \e^{3/2+\kappa}\mrd_\e \tilde \Psi^\e \|_{L^\infty_\e}$, together with \eqref{eq:bd_initial} yields \eqref{e:cubic_bd}. 
\end{proof}

Note that for $(p,q)=(3,0)$, \eqref{e:cubic-4} does not hold, so we need to exclude the term  $\mathfrak{C}^{3,0}$; however,
for the remainder term, this is not necessary and we have the following 
\begin{lemma}\label{lm:fifth_bd}
 Under Assumptions \ref{asmpt:v} and  \ref{as:models_abstract},
there exists a constant $C_M>0$ that only depends on initial data,  such that 
 \begin{equ}
  \label{eq:RemainderTermFifth}
 \norm{\e^{-2} \mrd_\e^*\widehat{\mathfrak{R}}(\e^2 \mrd_\e \widehat\theta^\e(t)) }_{\CC^{-1}_{\e}}
 \leq 
 C_M 
 \e^\kappa 
  (1+r + R_v)^5\,
  (1\wedge t)^{
 \frac{1+\eta-5\kappa}{2} \wedge \frac{5 \eta+3 -\kappa }{2} \wedge 0}.
\end{equ}
 \end{lemma}
\begin{proof}
Following Lemma \ref{lem:eps-improve-reg}\ref{pt:w_Deps} and the global bound $|\widehat{\mathfrak R}(x)|\lesssim |x|^5$ from Section~\ref{sec:derivation}, we have
\begin{align*}
\norm{\e^{-2} \mrd_\e^*\widehat{\mathfrak{R}}(\e^2 \mrd_\e \widehat\theta^\e(t)) }_{\CC^{-1}_\e} 
&\leq \norm{\e^{-2}\widehat{\mathfrak{R}}(\e^2 \mrd_\e \widehat\theta^\e(t)) }_{L^\infty_\e}
\\
&
 \leq \norm{\e^{8} (\mrd_\e \widehat\theta^\e)^5(t) }_{L^\infty_\e} \lesssim \sum_{p,q} \norm{R_{p,q}(t)}_{L^\infty_\e},
\end{align*}
where $p,q$ are nonnegative integers such that $0\leq p+q\leq 5$, and 
\[
R_{p,q} \eqdef \eps^{8} (\mrd_\eps \tilde \Psi^\e)^p (t)(\mrd_\e P^\e_t \theta^\e(0))^q (\mrd_\e \widehat v^\e)^{5-p-q}(t)\;.
\]
As in the proof of the last lemma, for $q\neq 0$, we introduce $\nu$ such that 
\begin{align*}
  \e^8 = \e^\kappa \left(\e^{3/2+\kappa}\right)^p \left(\e^{1+\nu}\right)^q \left(\e^{1-\alpha}\right)^{5-p-q}.
\end{align*}
For all $p \geq 0$, $q>0 $ satisfying $p+q\leq 5$, we have 
\begin{align*}
  \kappa+ p(3/2+\kappa) +(5-p-q)(1-\alpha) <8.
\end{align*}
Indeed, the left-hand side is increasing in $p$ and decreasing in $q$,
so even at $(p,q)=(5,0)$  the left-hand side is $6 \kappa +\frac{15}{2}<8$. This ensures $\nu\ge -1$ when $q>0$.

Applying  \eqref{assmp:bd_v}  for $\norm{ \e^{1-\alpha} \mrd_\e \widehat v^\e(t)} _{L^\infty_\e}$, \eqref {eq:bd_Psi_i} for $\| \e^{3/2+\kappa}\mrd_\e \tilde \Psi^\e \|_{L^\infty_\e}$, and \eqref{eq:bd_initial} for $ \norm{\mrd_\e P^\e_t \theta^\e_i(0)}^q_{L^\infty_\e}$, we have 
\begin{align*}
&\norm{R_{p,q}(t)}_{L^\infty_{\e}} \lesssim
 \e^\kappa r^p M^q R_v^{5-p-q} 
\begin{cases*}
(1\wedge t)^{\frac{1+\eta-5\kappa}{2} \wedge \frac{5 \eta+3 -\kappa }{2} \wedge 0} , & $q\geq 1$
\\
 (1\wedge t)^0, 
 & $q=0$.
\end{cases*}
\end{align*}
Indeed, applying \eqref{eq:bd_initial} on $ \norm{\mrd_\e P^\e_t \theta^\e_i(0)}^q_{L^\infty_\e}$ yields $(1\wedge t)^{q(\eta+\nu)/2 \wedge 0}$.
For each $q\ge 1$, this exponent is decreasing in $p$,
so in the worst case $p=5-q$, one has
\[
q(\eta+\nu)
\ge
\frac12-6\kappa
+
q\Big(\eta+\frac12+\kappa\Big).
\]
The minimum occurs at $q=1$ or $q=5$, which gives the exponent in \eqref{eq:RemainderTermFifth}. This concludes the proof.
\end{proof}

\subsection{Stochastic estimates}

Below we will write 
\[
\int_{\T_\eps^d} \; dy = \eps^d \sum_{y \in \T_\eps^d} \;,
\qquad 
\int_{\obonds_{\e,i}} \;dy=\eps^d \sum_{y \in \obonds_{\e,i}} \]
and we often identify $\obonds_{\e,i}$ for a fixed $i$ with $\T_\eps^d$.

Recall that for the solution of the  stochastic equation with zero initial condition, $\tilde\Psi$, we have 
\begin{equ}[e:Psi-formula]
\tilde\Psi_i^\e (t,x)
= \int_0^t \int_{\obonds_{\e,i}} P^\eps_{t-s} (x-y)\, 
 \xi^\eps_i(s,y)\,dy\,ds
\end{equ}
 where $P^\eps$ is the discrete heat kernel on $\obonds_{\eps,i}$. It is a singular kernel of order $-3$ in the sense that, uniformly in $t>0$, $\e>0$, and $|k|\geq 1$,
\begin{equ}[e:P-bound]
\begin{aligned}
|P^\e(t,x)|
&\lesssim 1+(|x|+\sqrt t+\e)^{-3},
\\
|(D^\e)^kP^\e(t,x)|
&\lesssim (|x|+\sqrt t+\e)^{-3-k},
\end{aligned}
\end{equ}
where $(D^\e)^k$ is a $k$-th spatial derivative (see \cite[Lemma~5.3]{HM18} and \cite[(6.12) and Appendix~B]{SSSX21}).
The first estimate includes the spatially constant mode of the torus heat kernel, while all the positive-order differences annihilate that mode and the remaining long-time part decays exponentially.
Our goal is to prove the stochastic estimates in Lemma~\ref{lm:ProbEstimate}.
We begin with some kernel bounds.

Fix $1\leq a<b\leq3$.
Recall that $\mrd_\e \tilde\Psi^\e$ is a $2$-form, defined on plaquettes.
Consider a plaquette $p=(x,\e_a,\e_b)$,
and write bonds $e_a= (x,x+\e_a)$ and $e_b=(x,x+\e_b)$. Define
\begin{equ}[e:def-X]
    X_\e(t,x)\eqdef
(\mrd_\e\tilde\Psi^\e)(t,p)
    =
    D^\e_a\tilde\Psi_b^\e(t,e_b)
    -
    D^\e_b\tilde\Psi_a^\e(t,e_a).
\end{equ}
Let 
\[
C^\e_{t,s}(x,y) =C^\e_{t,s}(x-y) \eqdef 
\E [X_\e (t,x)X_\e(s,y)],
\qquad
c_t^\e \eqdef C_{t,t}^\e(0).
\]
Clearly $C^\e_{t,s}(x-y) $ does not depend on $a,b$.
Since the discrete white noise and the discrete heat semigroup are
spatially translation invariant, $X_\e$ is stationary in space.
Consequently, $c_t^\e=\E[X_\e(t,x)^2]$ is independent of $x$.

\begin{lemma}\label{lem:kernel-C}
One has $C^\e_{t,s}(x,y) =C^\e_{t,s}(y,x)=C^\e_{s,t}(x,y)$.
Moreover, uniformly in $\eps$ and all the space and time variables,
\begin{equ}[e:C-bound]
\left|C^\e_{t,s}(x,y)\right|\lesssim 
(|x-y|+\sqrt {|t-s|}+\e)^{-3},
\end{equ}
and, uniformly in $\rho\in [0,1]$,
\begin{equ}[e:C-diff-bound]
| C_{t,t}^\e(x)
   -   C_{t,s}^\e(x)|
 \lesssim
    |t-s|^\rho ( |x| +\e)^{-3-2\rho}.
\end{equ}
\end{lemma}

\begin{proof}
Since $\tilde\Psi^\e_a$ and $\tilde\Psi^\e_b$ are independent,
following \eqref{e:Psi-formula},
one has
\[
C^\e_{t,s}(x-y)=\sum_{i\in \{a,b\}}   
 \int_0^{t\wedge s}
 \int_{\T_\e^3}
        D^\e_i P_{t-r}^\e(x-z)
D^\e_i P_{s-r}^\e(y-z)
\,dz\,dr .
\]
By summation by parts and  the semigroup property of the heat kernel, the above expression is equal to 
\[
 \sum_{i\in \{a,b\}}      \int_0^{t\wedge s}
        (D^\e_i)^*D^\e_i P_{t+s-2r}^\e(x-y)
\,dr 
\]
which is symmetric in $s,t$.
Without loss of generality, we assume $t\ge s$, then 
following a change of variable, this is equivalent to
\begin{equ}[e:tsts]
  \sum_{i\in \{a,b\}}     \frac12 \int_{t-s}^{t+ s}
        (D^\e_i)^*D^\e_i P_{u}^\e(x-y) \,du
\end{equ}
whose absolute value, by \eqref{e:P-bound}, is bounded by
\begin{equ}[e:integrate-u]
 \lesssim  \int_{t-s}^{\infty}
     (|x-y|+\sqrt {u}+\e)^{-5}    \,du
 \lesssim    
  (|x-y|+\sqrt {t-s}+\e)^{-3} .
\end{equ}
The last inequality follows directly from the change of variable $\sqrt u =v$.
This proves \eqref{e:C-bound}.

Regarding \eqref{e:C-diff-bound}, without loss of generality, we assume $t\ge s$. By \eqref{e:tsts}, 
\[
| C_{t,t}^\e(x)
   -   C_{t,s}^\e(x)|
=
 \Big| \sum_{i\in \{a,b\}}   
  \frac12
  \Big(\int_0^{2t} -  \int_{t-s}^{t+ s}\Big)
        (D^\e_i)^*D^\e_i P_{u}^\e(x) \,du
\Big|.
\]
Note that $[0,2t]\backslash [t-s,t+s] = [0,h]\cup [2t-h,2t]$ with $h=t-s$.
Using  \eqref{e:P-bound} as above, we obtain
\[
\int_0^h  (|x|+\sqrt {u}+\e)^{-5}    \,du
\le
\min \{h (|x|+\e)^{-5}  ,(|x|+\e)^{-3} \}\,,
\]
where the first term follows from setting $u=0$ which makes the integrand larger, and the second one
follows the same way as in \eqref{e:integrate-u}.
Since $\min\{A,B\}\le A^\rho B^{1-\rho}$,
\[
\int_0^h  (|x|+\sqrt {u}+\e)^{-5}    \,du
\le
    |t-s|^\rho ( |x| +\e)^{-3-2\rho}.
\]
For the integral over $[2t-h,2t]$, the same argument leads to the same bound
 (it is indeed more straightforward since $[2t-h,2t]$ does not contain the singularity of $P^\e$).
Therefore, 
$| C_{t,t}^\e(x)  -   C_{t,s}^\e(x)|$ is bounded by the right-hand side of \eqref{e:C-diff-bound}.
\end{proof}

We next study the behavior of 
$ c^\e_t  =\E[(\mrd_\e \tilde\Psi(t,p))^2] $, whose value does not depend on the specific plaquette $p$.

\begin{lemma} \label{lm:renom_cst}
One has
\begin{equ} \label{eq:cst_dPsisquare}
\eps^3 c^\e_t 
\lesssim 1.
\end{equ}
Moreover, uniformly in $\rho\in [0,1]$,
\begin{equ} \label{eq:ct-increment}
   \eps^3  | c_t^\e-c_s^\e |
    \lesssim
    |t-s|^{\rho}
    \e^{-2\rho}.
\end{equ}
\end{lemma}

\begin{proof}
Taking $x=y$ and $t=s$ in \eqref{e:C-bound}, 
we immediately obtain \eqref{eq:cst_dPsisquare}.
The second bound \eqref{eq:ct-increment} follows by 
writing 
\[
 | c_t^\e-c_s^\e |
\le
| c_t^\e - C^\e_{s,t}(0)|
+| C^\e_{s,t}(0)-c_s^\e |
\]
and applying \eqref{e:C-diff-bound} with $x=0$.
\end{proof}

The following lemma will justify Assumption \ref{as:models_abstract}.
\begin{lemma}
  \label{lm:ProbEstimate}
Let $T>0$, $p \geq 2$, and $0<\kappa<1$ sufficiently small. Then, for any $\tilde\kappa\in (0,\kappa)$,
\begin{align}
  \label{eq:SHE_moment}
 \E\sup_{t\in [0,T]}\|\tilde \Psi^\e(t)\|^p_{\CC^{-1/2-\kappa}_\e} & \le C_{p,T}  \;,
 \\
 \label{eq:third_mom}
  \E \sup_{t\in [0,T]} \norm{  \e^{4}\mrd_\e^* (\mrd_\e\tilde\Psi^\e )^3 (t)}^p_{\CC^{-3/2-2\kappa}_\e}
  & \le C_{p,T} \e^{p \tilde\kappa}.
\end{align}
Here, the constant $C_{p,T}>0$ may depend on $p,T,\kappa,\tilde\kappa$, but is independent of $\eps>0$.
\end{lemma}

\begin{proof} 
The first bound is rather classical, and we focus on the second bound.

By Lemma~\ref{lem:eps-improve-reg} and the convention on norms of
differential forms from Remark~\ref{rem:scalar-vector-convention},
\begin{equation}\label{eq:X_bound}
\norm{  \e^{4}\mrd_\e^* (\mrd_\e\tilde\Psi^\e)^3 (t)}_{\CC^{-3/2-2\kappa}_\e} 
\lesssim 
 \norm{  \e^{3+\kappa} (\mrd_\e\tilde\Psi^\e)^3 (t)}_{\CC^{-3/2-\kappa}_\e}. 
\end{equation}
As in \eqref{e:def-X}, we write $ X_\e = \mrd_\e\tilde\Psi^\e$ for brevity.

We first show that for  $t,s\in [0,T], z\in \T^3_\e$, $\lambda\in [\e,1]$ and any
rescaled test function $\varphi^\lambda_z$  as in \eqref{eq:f_e_f_alpha_neg}, one has uniformly in $\e>0$
\begin{align}
  \label{eq:3momentbd1}
\E \big|  \langle \e^{3+\kappa}X_\e^3(t),\; \phi^\lambda_z\rangle_\e \big|^2
& \lesssim \e^{2\kappa}\lambda^{-3-2\kappa},
 \\
  \label{eq:3momentbd2}
\E \big| \langle \e^{3+\kappa} (X_\e^3 (t)-X_\e^3(s)),\; \phi^\lambda_z\rangle_\e \big|^2
&\lesssim \lambda^{-3-2\kappa} |t-s|^{\kappa}.
\end{align}
Wiener decomposition gives
\[
X_\e(t,x)^3 \; =\; \Wick{X_\e(t,x)^3} + 3c_t^\e X_\e(t,x).
\]
By Wick's formula 
\[
\E [\Wick{X_\e(t,x)^3} \Wick{X_\e(t,y)^3}] = 6 C^\e_{t,t}(x-y)^3.
\]
By \eqref{e:C-bound} of Lemma~\ref{lem:kernel-C},  this kernel has singularity of order $-9$,
thus
\begin{equs}[e:3rd-chaos-mom]
\E \big| \langle \e^{3+\kappa}\Wick{X_\e^3}(t),\; \phi^\lambda_z\rangle_\e  \big|^2
&\lesssim 
 \e^{6+2\kappa} 
\int_{\T_\e^3}\int_{\T_\e^3}
 (|x-y|+\e)^{-9} |\phi^\lambda_z(x)\phi^\lambda_z(y) | \,dx\,dy
\\
&\lesssim 
 \e^{6+2\kappa} \e^{-6}\lambda^{-3}
 \leq 
 \e^{2\kappa}\lambda^{-3}.
\end{equs}
Here, the second inequality follows from the fact that the support of $\phi^\lambda$ has diameter $\lambda$
and the $L^\infty$ norm of $\phi^\lambda$ is bounded by $\lambda^{-3}$.

For the first chaos, 
\[
\E [X_\e(t,x) X_\e(t,y)] = C^\e_{t,t}(x-y).
\]
By \eqref{e:C-bound} of Lemma~\ref{lem:kernel-C},  this kernel has singularity of order $-3$.
By \eqref{eq:cst_dPsisquare} in Lemma~\ref{lm:renom_cst}, $c^\e_t \lesssim \e^{-3}$. 
Now, a subtlety is that in \eqref{e:3rd-chaos-mom}, 
although integrating the kernel of degree $-9$ yields
a divergent factor $\e^{-6}$, this divergence is absorbed by the outside factor $\e^{6+2\kappa}$.
However, for the first chaos, the similar calculation as in \eqref{e:3rd-chaos-mom} yields
 $$\e^{6+2\kappa} (c^\e_t )^2 \lesssim \e^{2\kappa},$$
which already accounts for the $\e^{2\kappa}$ 
required on the right-hand side of \eqref{eq:3momentbd1}. Hence no additional positive power of $\e$ remains to absorb  the logarithmic divergence arising from integration of the degree $-3$ kernel. We instead, exploit the fact that $X_\e$ is a derivative: 
  \begin{equation}\label{e:X-back-to-Psi}
    \begin{aligned}
\e^{2\kappa}
 \E \big| \langle  
 X_\e(t),\; \phi^\lambda_z  \rangle_\e 
\big|^2 
&\leq 
 \e^{2\kappa}\lambda^{-3-2\kappa}\E \norm{ X_\e(t)}_{\CC^{-3/2-\kappa}_\e}^2 \\
& \overset{\text{\ref{pt:w_D}}}{\lesssim} \e^{2\kappa} \lambda^{-3-2\kappa}\E \| \tilde{\Psi}(t)\|_{\CC^{-1/2-\kappa}_\e}^2  \overset{\eqref{eq:SHE_moment}}{\lesssim} \e^{2\kappa}\lambda^{-3-2\kappa}.
\end{aligned}
 \end{equation}
Together, we obtain \eqref{eq:3momentbd1}.

We turn to proving \eqref{eq:3momentbd2}. For the third chaos terms, following Wick's formula,
\begin{equ}
   \E[(\Wick{ X_\e^3}(t,x)\,\Wick{X_\e^3}(s,y))]
   = 6 C^\e_{t,s}(x-y)^3.
\end{equ}
Thus,
\begin{equs}
  \label{bd:thirdChaosTimeD}
\E \big| \langle 
\e^{3+\kappa} &
(\Wick{X_\e^3}(t)
 -
\Wick{X_\e^3}(s)), 
\; \phi^\lambda_z\rangle_\e \big|^2 
\\
& \lesssim
 \e^{6+2\kappa}  \int_{\T^3_\e}\int_{\T^3_\e}
 |J_3(x,y,t,s)|\,
 |\phi^\lambda_z(x)\,\phi^\lambda_z(y)|\, dx dy, 
\end{equs}
where
\[
J_3(x,y,t,s) 
 = C^\e_{t,t}(x,y)^3 -   C^\e_{t,s}(x,y)^3 -   C^\e_{s,t}(x,y)^3 + C^\e_{s,s}(x,y)^3 .
 \]
By \eqref{e:C-bound}  and  \eqref{e:C-diff-bound}  of Lemma~\ref{lem:kernel-C}, 
\begin{equs}
  |J_3(x,y,t,s)|
  &\lesssim
   |C_{t,t}(x,y)-C_{t,s}(x,y)| 
   \Big( C_{t,t}(x,y)^2 + C_{t,s}(x,y)^2  \Big) 
   \\
  & + |C_{s,s}(x,y)-C_{s,t}(x,y)|  \Big(C_{s,t}(x,y)^2 + C_{s,s}(x,y)^2\Big)
  \\
  &\lesssim
  \frac{|t-s|^{\kappa} }{(|x-y|+\e)^{9+2\kappa}}.
\end{equs}
Since $\e^{6+2\kappa} (|x-y|+\e)^{-9-2\kappa} \le \e^{2\kappa} (|x-y|+\e)^{-3-2\kappa}$, 
\begin{equs}
\eqref{bd:thirdChaosTimeD}
&\lesssim
\e^{2\kappa}|t-s|^{\kappa}
  \int_{\T^3_\e}\int_{\T^3_\e}
(|x-y|+\e)^{-3-2\kappa}
|\phi^\lambda_z(x)\,\phi^\lambda_z(y)|\, dx dy
\\
&\lesssim
\e^{2\kappa}|t-s|^{\kappa}
\e^{-2\kappa}
\lambda^{-3}
=
|t-s|^{\kappa} \lambda^{-3}\,.
\end{equs}
Here, the second inequality follows from the fact that the support of  $\phi^\lambda$ is of diameter $\lambda$
and $\phi^\lambda$ is bounded by $\lambda^{-3}$ in $L^\infty$.

For the first chaos terms,  note that 
\[
c^\e_t X_\e(t) - c^\e_s X_\e(s) 
=
c^\e_t (X_\e(t) - X_\e(s))
+
(c^\e_t - c^\e_s ) X_\e(s),
\]
thus, 
\begin{equs}
\E \Big| \Big\langle \e^{3+\kappa}  
\Big(c^\e_t X_\e(t) &- c^\e_s X_\e(s) \Big),\; \phi^\lambda_z  
\Big\rangle_\e 
\Big|^2   
\lesssim
\e^{6+2\kappa} 
    (c_t^\e)^2
    \E
    \left|
        \left\langle
            X_\e(t)-X_\e(s),
            \varphi_z^\lambda
        \right\rangle_\e
    \right|^2
 \\
&+\e^{6+2\kappa} |c_t^\e-c_s^\e|^2
    \mathbf E
    \left|
        \left\langle
            X_\e(s),
            \varphi_z^\lambda
        \right\rangle_\e
    \right|^2
= I_1 + I_2.
\end{equs}
By \eqref{eq:cst_dPsisquare} in Lemma~\ref{lm:renom_cst},
\begin{equ}
  \label{bd:firstChaosTimeD}
I_1\lesssim
\e^{2\kappa} \int_{\T^3_\e}\int_{\T^3_\e} |J_1(x,y,t,s)\,\phi^\lambda_z(x)\,\phi^\lambda_z(y)|\, dx dy,
\end{equ}
where
\begin{equ}
    J_1(x,y,t,s)
    = C^\e_{t,t}(x,y) - C^\e_{t,s}(x,y) - C^\e_{s,t}(x,y) + C^\e_{s,s}(x,y).
\end{equ}
Following \eqref{e:C-diff-bound}  of Lemma~\ref{lem:kernel-C}, 
\[
| J_1(x,y,t,s) |
 \lesssim
    |t-s|^{\kappa} ( |x-y| +\e)^{-3-2\kappa}.
\]
Therefore, 
\[
I_1 \lesssim 
|t-s|^{\kappa}  \lambda^{-3}.
\]
By \eqref{eq:ct-increment} in Lemma~\ref{lm:renom_cst} with $\rho=\kappa/2$,
and \eqref{e:X-back-to-Psi},
\[
I_2 \lesssim
|t-s|^{\kappa} \lambda^{-3-2\kappa}.
\]
Combining these bounds, we obtain  \eqref{eq:3momentbd2}.

Interpolating \eqref{eq:3momentbd1} and \eqref{eq:3momentbd2} and using hypercontractivity (i.e., equivalence of  moments for variables in finite chaos), we obtain for any $p\geq 1$, uniformly in 
$\tilde\kappa \in [0,\kappa]$,
\begin{equ}
  \label{eq:3momentbd3}
\E \big| \langle \e^{3+\kappa} (X_\e^3 (t)-X_\e^3(s)),\; \phi^\lambda_z\rangle_\e \big|^p
\lesssim (\e^{\tilde\kappa} |t-s|^{(\kappa-\tilde\kappa)/2}\lambda^{-\frac{3}{2}-\kappa})^p
\;.
\end{equ}
By wavelet characterization of $\CC_\e^{-3/2-\kappa}$ 
(\cite[Proposition~3.20, Theorem~10.7]{Hairer14} and \cite[Proposition~4.9]{HM18})
and a Kolmogorov type argument, the bound \eqref{eq:3momentbd3} implies that for any $\hat\kappa>\kappa$,
\begin{equation*}
\E \|\e^{3+\kappa} X^3_\e(t)-\e^{3+\kappa} X^3_\e(s)\|^p_{\CC^{-3/2-\hat\kappa}_\e}
\lesssim \e^{p\tilde\kappa} |t-s|^{p(\kappa-\tilde\kappa)/2}
\end{equation*}
The classical Kolmogorov continuity criterion  
(in the form of \cite[Theorem A.10]{Friz_Victoir_2010}) now
implies that $\E \sup_{t\in [0,T]} \norm{  \e^{3+\hat\kappa} X_\e^3 (t)}^p_{\CC^{-3/2-\hat\kappa}_\e}
\lesssim \e^{p(\tilde\kappa + \hat\kappa-\kappa)}\lesssim \e^{p\tilde\kappa}$ for any $\tilde\kappa<\kappa<\hat\kappa$.
Since there are only three components of a $2$-form, the same estimate holds with $X^3_\e$ replaced by $(\mrd_\e\tilde\Psi)^3$.
Combined with \eqref{eq:X_bound}, applied with $\hat\kappa$ in place of $\kappa$,
we obtain \eqref{eq:third_mom} after relabelling $\hat\kappa$ as $\kappa$.
\end{proof}

For every $r>0$, $\e>0$, we define 
\begin{equ}
  \label{eq:GoodEvent}
\Omega_{r,\eps}=\Omega_{r,\eps}(T) \eqdef
 \{ \omega \in \Omega \; |\; \text{Assumption \ref{as:models_abstract} holds with } r \}
\end{equ}
Following Lemma \ref{lm:ProbEstimate} with $\tilde\kappa =\frac{\kappa}{2}$  and Markov's inequality, 
there exists a constant $C$ independent of $\eps$ such that  
\begin{align}
   \P(\Omega_{r,\eps}^c) 
   & \leq  
    \P( 
\sup_{t\in [0,T]}\|\tilde \Psi^\e(t)\|_{\CC^{-1/2-\kappa}_\e} \geq r) \nonumber \\ 
&\qquad 
+ \P(
  \sup_{t\in [0,T]} \norm{  \e^{4}\mrd_\e^* (\mrd_\e\tilde\Psi^\e )^3 (t)}_{\CC^{-3/2-2\kappa}_\e}
 \geq  r \e^{\frac{\kappa}{2} }) \leq \frac{C_{p,T}}{r^p}
\end{align}
for every $p\ge 2$.
Namely,
\begin{equ}
  \label{eq:prob_Omega}
    \P(\Omega_{r,\eps}) \geq 1-\frac{C_{p,T}}{r^p}
\end{equ}
uniformly in $\eps>0$.

\subsection{A priori bound on the remainder}

Define $\widehat\tau_\e$ similarly to \eqref{e:def-tau}: 
\begin{equ}[e:def-tauh]
    \widehat\tau_\eps
    \eqdef
    \inf\left\{
        t\ge0:
        \|\eps \widehat\theta^\e(t)\|_{L^\infty} \ge \frac{\rho}{4}
    \right\}.
\end{equ}

The goal of this subsection is to show that on $\Omega_{r,\e}$,
the remainder $\widehat v^\e$ is bounded in $\CC_\e^{4\kappa}$ by $\e^{\kappa/2}$ 
on $[0,T]$, and for small enough $\eps$, we have $\widehat\tau_\e>T$,
namely, the Lie algebra valued process does not exit the $\rho$-neighborhood before $T$.

\begin{proposition}
\label{prop:remainder}
Fix $T>0$ and $M,r\ge1$.
Suppose that, almost surely,
\begin{equ}
  \|\theta^\e(0)\|_{\mathcal C^\eta_\e}
  \le M,
 \qquad
  -\frac23+\frac53\kappa < \eta.
\end{equ}
There exist constants $C_{M,T}<\infty$ and
$\e_0=\e_0(M,T, r)>0$ such that for every
$\e\le\e_0$, on the event $\Omega_{r,\e}(T)$ from \eqref{eq:GoodEvent} we have
\begin{equ}[e:no-exit]
  \widehat\tau_\e>T
\end{equ}
and
\begin{equ}[eq:v-apriori-final]
  \|\widehat v^\e\|_
  {\mathcal C^{4\kappa,T}_{0,\e}}
  \le
  C_{M,T}\e^{\kappa/2}r^5.
\end{equ}
Consequently,
$\widehat \theta^\e$ agrees with the original
Langevin process $\theta^\e$ on $[0,T]$ and $\tau_\e>T$ on $\Omega_{r,\e}(T)$.
\end{proposition}

\begin{proof}
For $t\in[0,T]$, write
\[
    \|\widehat v^\e\|_{4\kappa,t}
    \eqdef
    \bigl\|
        \widehat v^\e
    \bigr\|_{\mathcal C^{4\kappa,t}_{0,\e}}
    =
    \sup_{s\in[0,t]}
    \|\widehat v^\e(s)\|_{\mathcal C^{4\kappa}_\e}.
\]
Define the stopping time
\[
    \sigma_\e
    \eqdef
    T\wedge
    \inf\left\{
        t\in[0,T]:
        \|\widehat v^\e\|_{4\kappa,t}\geq1
    \right\},
\]
with the convention that the infimum of the empty set is $+\infty$.
Since $\widehat v^\e(0)=0$
 and 
$\widehat v^\e$ is continuous in time for fixed $\e$, one has
$\sigma_\e>0$.

For any  $ t\in [0,\sigma_\e]$,
\[
    \|\widehat v^\e\|_{4\kappa,t}
    \leq1,
\]
assumption~\ref{asmpt:v} is thus satisfied on the interval  $[0,\sigma_\e]$ with $\alpha = 4\kappa$ and $R_v = 1 \leq r$.

We write the equation \eqref{eq:v_i} in the mild form: 
\begin{equs} [eq:v-mild-three-terms]
    \widehat v^\e(t)
    &=
    \int_0^t
    P_{t-s}^\e
    \Bigl[
        a_{3} \e^4 \mrd_\e^*
        \bigl(
            \mrd_\e\widetilde\Psi^\e
        \bigr)^3(s)
    \Bigr]\,ds
    \\
    &\quad+
    \sum_{(p,q)\neq (3,0)} a_3   \binom{3}{p} \binom{3-p}{q}
     \int_0^t
    P_{t-s}^\e
    \mathfrak C^{p,q}(s)\,ds
    +
    \int_0^t
    P_{t-s}^\e
    \mathcal R^\e(s)\,ds
\end{equs}
where $\mathfrak C^{p,q}$ is defined in \eqref{eq:Cpq} for $p,q\ge 0$ with $p+q\leq 3$, and
\[
\mathcal R^\e(s)
=
\e^{-2}\mrd_\e^*
\widehat{\mathfrak R}\big(\e^2\mrd_\e\widehat\theta^\e(s)\big).
\]

On $\Omega_{r,\e}(T)$, using the uniform bound on $a_3$ in \eqref{e:coeff-bounds}, the first term satisfies \eqref{e:assump-Psi-2}, namely
\begin{equ}[eq:pure-cubic-bound-proof]
    \left\|
        a_{3}\e^4 \mrd_\e^*
        \bigl(
            \mrd_\e\widetilde\Psi^\e
        \bigr)^3(s)
    \right\|_{\mathcal C^{-3/2-2\kappa}_\e}
    \leq C_M
    r\e^{\kappa/2}.
\end{equ}
By Lemma~\ref{lm:cubic_bd},
\begin{equ}
\label{eq:mixed-bound-proof}
   \sum_{(p,q)\neq (3,0)}   \|\mathfrak C^{p,q}(s)\|_{\CC^{-1}_\e}
    \leq
    C_M\e^\kappa r^3
    (1\wedge s)^{b_3},
\end{equ}
where
\[
    b_3
    \eqdef
    \frac{1+3\eta-\kappa}{2}
    \wedge
    \frac{\eta-3\kappa}{2}\wedge 0.
\]
Similarly, by Lemma~\ref{lm:fifth_bd},
\begin{equation}
\label{eq:remainder-bound-proof}
    \|\mathcal R^\e(s)\|_{\mathcal C^{-1}_\e}
    \leq
    C_M\e^\kappa r^5
    (1\wedge s)^{b_5},
\end{equation}
where
\[
    b_5
    \eqdef
    \frac{1+\eta-5\kappa}{2}
    \wedge
    \frac{5\eta+3-\kappa}{2}
    \wedge
    0.
\]
The constant $C_M$ does not depend on
$\e$, $r$, $s$, or $\sigma_\e$.

The discrete Schauder estimate gives, uniformly in $\e$,
\begin{equ}
\label{eq:schauder-proof}
    \|P_s^\e f\|_{\CC^\alpha_\e}
    \lesssim
    (s\wedge1)^{-(\alpha-\beta)/2}
    \|f\|_{\CC^\beta_\e},
    \qquad s>0.
\end{equ}
Applying \eqref{eq:schauder-proof} with $\alpha=4\kappa$, then 
$\beta=-3/2-2\kappa$ to the first term in
\eqref{eq:v-mild-three-terms}, and $\beta=-1$ to the other two, we have for $t\leq\sigma_\e$,
\begin{equs}[eq:v-integrals]
    \|\widehat v^\e(t)\|_{\CC^{4\kappa}_\e}
    &\leq C_M
    \e^{\kappa/2}r
    \int_0^t
    ((t-s)\wedge1)^{-\frac34-3\kappa}
    \,ds
    \notag\\
    &\quad+
    C_M\e^\kappa r^3
    \int_0^t
    ((t-s)\wedge1)^{-\frac{1+4\kappa}{2}}
    (1\wedge s)^{b_3}\,ds
    \notag\\
    &\quad+
    C_M\e^\kappa r^5
    \int_0^t
    ((t-s)\wedge1)^{-\frac{1+4\kappa}{2}}
    (1\wedge s)^{b_5}\,ds.
\end{equs}
We now bound these three integrals, uniformly for
$t\in[0,T]$. Since $\kappa>0$ is sufficiently small, the first integral is bounded by a constant that depends only on $T$.
For the second and the third integrals, one can check using $\eta > -\frac23+\frac53\kappa $ that
\begin{equ}[e:above-1]
 -\frac{1+4\kappa}{2}+b_3 > -1 ,\qquad  - \frac{1+4\kappa}{2} +b_5 > -1.
\end{equ}
Recall,for instance using the Beta function, that whenever  $a<1$ and $b> -1$,
\[
    \int_0^t
    (t-s)^{-a}s^b\,ds
    \lesssim 
    C_{a,b} t^{1-a+b} \mathbf 1_{t\le 1}+ C_T \mathbf 1_{t\in [1,T]}.
\]
Substituting this bound in \eqref{eq:v-integrals}  under the condition \eqref{e:above-1}
yields positive powers of $t$. Thus,  \eqref{eq:v-integrals} implies
\begin{equ}[eq:v-bootstrap-bound]
    \|\widehat v^\e\|_
    {4\kappa,\sigma_\e}
    \leq
    C_{M,T}
    \Big(
        \e^{\kappa/2}r
        +
        \e^\kappa r^3
        +
        \e^\kappa r^5
    \Big)
    \leq
    C_{M,T}\e^{\kappa/2}r^5,
\end{equ}
where in the last inequality we used $r\geq1$ and
$\e\leq1$.

Choose $\e_0=\e_0(M,T, r)$ sufficiently small such that
\begin{equ}[e:choose-e0]
    C_{M,T}\e^{\kappa/2}r^5<1/2\;, \qquad \forall  \e\leq\e_0 \;.
\end{equ}
If $\sigma_\e<T$, then continuity and the definition of
$\sigma_\e$ would imply that
$\|\widehat v^\e\|_{4\kappa,\sigma_\e}=1$;
but this contradicts \eqref{eq:v-bootstrap-bound} and \eqref{e:choose-e0}. Hence
$
    \sigma_\e=T
$
and \eqref{eq:v-apriori-final} follows.

It remains to prove \eqref{e:no-exit}, i.e.,  $\widehat\theta^\e$ does not exit the
 $\rho$-neighborhood. Recall that
\[
    \widehat\theta^\e(t)
    =
    \widetilde\Psi^\e(t)
    +
    P_t^\e\theta^\e(0)
    +
    \widehat v^\e(t).
\]
By Lemma \ref{lem:eps-improve-reg}\ref{pt:wo_D} and \eqref{eq:v-apriori-final}, on $\Omega_{r,\e}(T)$,
\begin{equs}
 & {}  \sup_{t\in[0,T]}
    \e
    \|\widetilde\Psi^\e(t)\|_{L^\infty_\e}
     \lesssim
    \e^{1/2-\kappa}
    \sup_{t\in[0,T]}
    \|\widetilde\Psi^\e(t)\|_
    {\CC^{-1/2-\kappa}_\e}
  \lesssim
    \e^{1/2-\kappa}r,
  \\
 &   \sup_{t\in[0,T]}
    \e
    \|P_t^\e\theta^\e(0)\|_
    {L^\infty_\e}
    \lesssim
    \e^{1+\eta}
    \sup_{t\in[0,T]}
    \|P_t^\e\theta^\e(0)\|_
    {\CC^\eta_\e}
   \lesssim
    \e^{1+\eta}M,
\\
   & \sup_{t\in[0,T]}
    \e
    \|\widehat v^\e(t)\|_{L^\infty_\e}
    \lesssim
    \e
    \|\widehat v^\e\|_{4\kappa,T}
   \leq
    C_{M,T}
    \e^{1+\kappa/2}r^5.
\end{equs}
Combining these bounds, we find
\begin{equ}
\label{eq:theta-Linfty-final}
    \sup_{t\in[0,T]}
    \e
    \|\widehat\theta^\e(t)\|_{L^\infty_\e}
    \leq
    C\e^{1/2-\kappa}r
    +
    C\e^{1+\eta}M
    +
    C_{M,T}\e^{1+\kappa/2}r^5.
\end{equ}
Since $\eta>-1$, the right-hand side tends to zero as
$\e\to0$, with $M$, $T$, and $r$ fixed. Decreasing
$\e_0$ if necessary, we may therefore assume that
\[
    \sup_{t\in[0,T]}
    \e
    \|\widehat\theta^\e(t)\|_{L^\infty_\e}
    <
    \frac{\rho}{4}.
\]
By the definition of the exit time $\widehat\tau_\e$, this proves   \eqref{e:no-exit}, i.e. $ \widehat\tau_\e>T$.

We conclude that $\widehat\theta^\e$ and the original $\theta^\e$ agree on $[0,T]$ and $\tau_\e>T$.
Indeed, on $[0,\tau_\e \wedge \widehat\tau_\e]$,  $\theta^\e$ and $\widehat\theta^\e$ satisfy exactly the same equation
with the same initial condition and noise, so by uniqueness they are equal;
since $\widehat\tau_\e>T$, $\theta^\e$ cannot exit the same $\rho$-neighborhood before $T$.
Thus, $\tau_\e>T$ and $\widehat\theta^\e=\theta^\e$ on $[0,T]$.
\end{proof}

\begin{corollary}
\label{cor:remainder-and-exit}
Assume  that  $\sup_{\e\in(0,1]}  \|\theta^\e(0)\|_{\mathcal C^\eta_\e}\le M$ almost surely for some deterministic $M<\infty$.
Then, for every fixed  $T>0$ and every $\delta>0$,
\[
    \lim_{\e\to0}
    \mathbf P\left(
        \tau_\e\leq T
        \ \text{or}\
        \|\widehat v^\e\|_
        {\mathcal C^{4\kappa,T}_{0,\e}}
        >\delta
    \right)
    =
    0.
\]
In particular, as $\e\to 0$,
\[
    \mathbf P(\tau_\e\leq T) \to 0
\]
and
\[
    \|\widehat v^\e\|_{\CC^{4\kappa,T}_{0,\e}}
    \to 0
    \quad\text{in probability}.
\]
\end{corollary}

\begin{proof}
Fix an arbitrary $r\geq1$. By
Proposition~\ref{prop:remainder}, on
$\Omega_{r,\e}(T)$ and for all 
$\e<\e_0(M,T,r)$,
\[
    \tau_\e>T
\qquad
\mbox{and}
\qquad
    \|\widehat v^\e\|_
    {\mathcal C^{4\kappa,T}_{0,\e}}
    \leq
    C_{M,T}\e^{\kappa/2}r^5.
\]
For fixed $r$, the right-hand side is smaller than $\delta$ for all
sufficiently small $\e$. Therefore,
\[
\begin{aligned}
    &\limsup_{\e\to0}
    \mathbf P\left(
        \tau_\e\leq T
        \ \text{or}\
        \|\widehat v^\e\|_
        {\mathcal C^{4\kappa,T}_{0,\e}}
        >\delta
    \right)
    \\
    &\qquad\leq
    \limsup_{\e\to0}
    \mathbf P\bigl(
        \Omega_{r,\e}(T)^c
    \bigr)
    \leq
    C_{p,T}r^{-p}.
\end{aligned}
\]
Since $r$ is arbitrary, letting $r\to\infty$ proves the result.
\end{proof}

\subsection{Convergence of the linear equation}

Recall that in Theorem \ref{thm:main},  we choose $\eta<0<\kappa$.
Denote $P_t=e^{t\Delta}$ and, as earlier, $P_t^\e=e^{t\Delta_\e}$.
As in the discussion following Assumption~\ref{asmpt:v},
for $i=1,2,3$, write
\begin{equation*}
\Psi_i(t)=P_t\theta_i(0)+\widetilde\Psi_i(t)\;,
\qquad
\Psi_i^\e(t)=P_t^\e\theta_i^\e(0)+\widetilde\Psi_i^\e(t)\;,
\end{equation*}
where
\begin{equation}\label{eq:tilde_Psi_def}
\widetilde\Psi_i(t)=\int_0^t P_{t-s}\xi_i(s)\,ds\;,
\qquad
\widetilde\Psi_i^\e(t)=\int_0^t P_{t-s}^\e\xi_i^\e(s)\,ds\;.
\end{equation}

All statements below are understood componentwise, with the norm of a $1$-form given by the maximum over its three components as in Remark~\ref{rem:scalar-vector-convention}.

\begin{lemma}
\label{lem:initial-semigroup-convergence}
Denote
\[
    \alpha=-\frac12-\kappa\;,
    \qquad
    \gamma=\Big(\eta+\frac12\Big)\wedge 0\;.
\]
Suppose that, as $\eps\downarrow0$,
\[
    \sup_{\e\in(0,1]}
    \|f^\e\|_{\CC^\eta_\e}
    +
    \|f\|_{\CC^\eta}
    <\infty,
    \qquad
    \|f^\e;f\|_{\CC^\eta_\e}\to0.
\]
Then, for every $T>0$, as $\eps\downarrow0$,
\[
    \|P^\e_\cdot f^\e;P_\cdot f\|
    _{\CC^{\alpha,T}_{\gamma,\e}}
    \to0 .
\]
\end{lemma}

\begin{proof}
Recall the discrete and continuous Schauder estimates, uniformly in
$\e$,
\begin{equation}\label{eq:Schauder}
    \sup_{t\in(0,T]}
    (t^{1/2}\wedge1)^{-\gamma}
    \|P_t^\e f^\e\|_{\CC^\alpha_\e}
    \lesssim_T
    \|f^\e\|_{\CC^{\alpha+\gamma}_\e},
\end{equation}
and analogously for $P_t f$.

Note that
\[
\eta - \alpha - \gamma \geq \kappa>0\;.
\]
Mollifying $f^\e$ and $f$ at scale $\bar\e\in(\e,1]$ yields
$f_{\bar\e}^\e \in \CC^{\eta}_\e$ and smooth $f_{\bar\e}\in\CC^\eta$ satisfying, uniformly in $\e\in (0,\bar\e)$,
\begin{equation}\label{eq:mollification_error}
    \|f^\e-f_{\bar\e}^\e\|_{\CC^{\alpha+\gamma}_\e}
    \lesssim
    \bar\e^{\eta-\alpha-\gamma} \norm{f^\e}_{\CC^\eta_\e},
    \quad
    \|f-f_{\bar\e}\|_{\CC^{\alpha+\gamma}}
    \lesssim
    \bar\e^{\eta-\alpha-\gamma} \norm{f}_{\CC^\eta}
    \;.
\end{equation}
These are standard mollification estimates; for the second bound, writing $\rho_{\bar\e}$ for the mollifier, we have  $\langle f-f_{\bar\e},\phi_z^\lambda\rangle=\int \rho_{\bar\e} (y) \langle f,\phi_z^\lambda-\phi_{z-y}^\lambda\rangle dy$. If $\lambda\le \bar\e$ we bound the two terms separately by $\lesssim \lambda^\eta$. If $\lambda\ge \bar\e$ , the difference of test functions is of size $O(\bar\e/\lambda)$.  In either case, the resulting expression is bounded by $ \bar\e^{\eta-\alpha-\gamma} \lambda^{\alpha+\gamma}\norm{f}_{\CC^\eta}$.
The first (discrete) estimate is proved identically, with the integral replaced  by the normalized sum.

For fixed $\bar\e>0$,
the convergence $\|f^\e;f\|_{\CC^\eta_\e}\to0$ implies
\[
\lim_{\e\downarrow0}\sup_{x\in\T^3_\e}|f^\e_{\bar\e}(x) - f_{\bar\e}(x)| = 0\;.
\]
Therefore, standard convergence of the
centered finite-difference approximation of the heat equation implies
\[
  \lim_{\eps\downarrow0}  \|P_\cdot^\e f_{\bar\e}^\e;P_\cdot f_{\bar\e}\|
    _{\CC^{\alpha,T}_{\gamma,\e}} = 0.
\]
(The precise choices of $\alpha$ and $\gamma$ are not important for this step as long as $\alpha<0$ and $\gamma\leq0$.
Indeed, for a smooth function $u$, the Riemann-sum error against a test function $\varphi_z^\lambda$ is bounded by $C_u\e\lambda^{-1}$, so after multiplication by $\lambda^{-\alpha}$ with $\alpha\in (-1,0)$ and using $\lambda\in[\e,1]$, it is bounded by $C_u \e^{-\alpha}$, which tends to zero since $\alpha<0$.)

On the other hand, \eqref{eq:Schauder} and \eqref{eq:mollification_error} imply
that, uniformly in $\e<\bar\e$,
\[
 \|P_\cdot^\e f_{\bar\e}^\e - P_\cdot^\e f^\e\|
    _{\CC^{\alpha,T}_{\gamma,\e}} \lesssim \|f^\eps-f^\eps_{\bar\e}\|_{\CC^{\alpha+\gamma}_\e}
    \lesssim \bar\e^{\eta-\alpha-\gamma}\;,
\]
and similarly for $f$ and $f_{\bar\e}$. Since $\eta-\alpha-\gamma>0$, the right-hand side can be made arbitrarily small by choosing $\bar\e$ sufficiently small.

By the definition of the $\CC^{\alpha,T}_{\gamma,\e}$ norm, we have
\[
\|P^\e_\cdot f^\e;P_\cdot f\|_{\CC^{\alpha,T}_{\gamma,\e}}
\leq
 \|P_\cdot^\e f_{\bar\e}^\e - P_\cdot^\e f^\e\|_{\CC^{\alpha,T}_{\gamma,\e}}
 + \|P_\cdot^\e f_{\bar\e}^\e;P_\cdot f_{\bar\e}\|_{\CC^{\alpha,T}_{\gamma,\e}}
 + \|P_\cdot f_{\bar\e} - P_\cdot f\|_{\CC^{\alpha,T}_{\gamma}}.
\]
Choose $\bar\e>0$ small so that the first and third terms are small uniformly in $\e$. Then let $\eps\downarrow0$ so that the middle term is small. This completes the proof.
\end{proof}

\begin{lemma}
\label{lem:stochastic-convolution-convergence}
Consider $\kappa>0$, $T>0$, and $\gamma\le0$.
Then
\[
\|\widetilde\Psi^\e;\widetilde\Psi\|_{\CC^{-1/2-\kappa,T}_{\gamma,\e}}
\to 0
\qquad\text{in probability.}
\]
\end{lemma}

\begin{proof}
Denote by $\star$, $\star_\e$ the convolutions on $\R \times \T^3$ and $\R \times \T^3_\e$ respectively.
Recall the coupling in \eqref{eq:Disc_white_noise}. Let $\psi\in\CC_c^\infty(\R\times\R^3)$ be a parabolic mollifier of integral $1$, and set
\[
\psi^{\bar\e}(t,x)
=\bar\e^{-5}\psi(\bar\e^{-2}t,\bar\e^{-1}x),
\qquad
\xi^{\bar\e,0}=\xi\star\psi^{\bar\e}.
\]
For $x\in\T_\e^3$, define the cell $Q_x^\e=x+[-\e/2,\e/2)^3$ in $\T^3$ and the averages
\begin{equ}
\label{eq:DiscreteMollifier}
\psi^{\bar\e,\e}(t,x)
=\e^{-3}\int_{Q_x^\e}\psi^{\bar\e}(t,y)\,dy,
\qquad
\xi^{\bar\e,\e}=\xi^\e\star_\e\psi^{\bar\e,\e}.
\end{equ}
The cells $\{Q_x^\e\}_{x\in \T^3_\e}$ form a partition, so
\[
\int_{\R\times\T_\e^3}\psi^{\bar\e,\e}(t,x)\,dt\,dx=1\;.
\]
Let $\widetilde\Psi^{\bar\e,0}$ and $\widetilde\Psi^{\bar\e,\e}$ denote the stochastic convolutions as in \eqref{eq:tilde_Psi_def} but driven by $\xi^{\bar\e,0}$ and $\xi^{\bar\e,\e}$ respectively, both with zero initial data.

A standard argument using It\^o isometry, the mollification estimate
\cite[Lemma~7.6]{HM18},
and a Kolmogorov bound,
implies that there exist $p\ge2$ and $q>0$ such that, uniformly in $0<\e<\bar\e\leq 1$,
\begin{align}
\label{eq:linear-mollification-errors}
\E\|\widetilde\Psi^\e-\widetilde\Psi^{\bar\e,\e}\|_{L^\infty([0,T],\CC^{-1/2-\kappa}_\e)}^p
&\lesssim_T \bar\e^{qp}\;,
\\
\E\|\widetilde\Psi-\widetilde\Psi^{\bar\e,0}\|_{L^\infty([0,T],\CC^{-1/2-\kappa})}^p
&\lesssim_T \bar\e^{qp}\;.
\end{align}
For every fixed $\bar\e>0$, the noises are smooth and
\[
\lim_{\eps\downarrow0}\sup_{|t|<R}\sup_{x\in\T^3_\e}|\xi^{\bar\e,\e}(t,x) - \xi^{\bar\e,0}(t,x)| =0
\]
in probability for any $R>0$ (see e.g. the proof of \cite[Lemma 7.4]{chevyrev2023invariant} for more details).
Thus, standard results in numerical analysis, together with the
Riemann-sum estimate used in the proof of
Lemma~\ref{lem:initial-semigroup-convergence}, imply that
\begin{equ}
\label{eq:smooth-linear-consistency}
\sup_{t \in [0,T]}\|\widetilde\Psi^{\bar\e,\e}(t);\widetilde\Psi^{\bar\e,0}(t)\|_{\CC^{-1/2-\kappa}_\e}
\to 0
\qquad\text{in probability as }\e\to0.
\end{equ}
Since $\gamma\leq 0$, we obtain
\begin{align*}
\|\widetilde\Psi^\e;\widetilde\Psi\|_{\CC^{-1/2-\kappa,T}_{\gamma,\e}}
&\le
\|\widetilde\Psi^\e-\widetilde\Psi^{\bar\e,\e}\|_{L^\infty([0,T],\CC^{-1/2-\kappa}_\e)}
\\
&\quad+
\sup_{t \in [0,T]}\|\widetilde\Psi^{\bar\e,\e}(t);\widetilde\Psi^{\bar\e,0}(t)\|_{\CC^{-1/2-\kappa}_\e}
+
\|\widetilde\Psi^{\bar\e,0}-\widetilde\Psi\|_{L^\infty([0,T],\CC^{-1/2-\kappa})}.
\end{align*}
Now choose $\bar\e$ small so that the first and third terms are small in probability uniformly in $\e<\bar\e$ by \eqref{eq:linear-mollification-errors}, and then let $\e\to0$ so that the middle term goes to zero by \eqref{eq:smooth-linear-consistency}. This proves the claim.
\end{proof}

\begin{theorem}
\label{theo:Conv_SHE}
Let $-\frac23<\eta<0$, fix $\kappa>0$ sufficiently small, and set
\[
\gamma=\bigl(\eta+\tfrac12\bigr)\wedge 0\;.
\]
Suppose that there is a deterministic $M<\infty$ such that, almost surely,
\[
\sup_{0<\e\le1}\|\theta^\e(0)\|_{\CC^\eta_\e}
+\|\theta(0)\|_{\CC^\eta}
\le M,
\]
and suppose that
\[
\|\theta^\e(0);\theta(0)\|_{\CC^\eta_\e}\to 0
\qquad\text{in probability}.
\]
Then, for every $T>0$ and $\delta>0$,
\[
\lim_{\e\to0}
\P\left(
\|\Psi^\e;\Psi\|_{\CC^{-1/2-\kappa,T}_{\gamma,\e}}>\delta
\right)=0.
\]
\end{theorem}

\begin{proof}
Using the decomposition into initial condition terms and stochastic convolutions,
\begin{equation*}
\|\Psi^\e;\Psi\|_{\CC^{-1/2-\kappa,T}_{\gamma,\e}}
\le
\|P_\cdot^\e\theta^\e(0);P_\cdot\theta(0)\|_{\CC^{-1/2-\kappa,T}_{\gamma,\e}}
+
\|\widetilde\Psi^\e;\widetilde\Psi\|_{\CC^{-1/2-\kappa,T}_{\gamma,\e}}.
\end{equation*}
Lemma~\ref{lem:initial-semigroup-convergence}, together with the uniform bound by $M$ and the assumed convergence of the initial conditions yields that the first term converges to zero in probability. The second term converges to zero in probability by Lemma~\ref{lem:stochastic-convolution-convergence}.
\end{proof}

\begin{proof}[of Theorem~\ref{thm:main}]
On the event $\{\tau_\e>T\}$, the original solution agrees on $[0,T]$ with the globally defined process $\widehat \theta^\e$ and thus
\[
\theta^\e=\Psi^\e+\widehat v^\e.
\]
Since $4\kappa>-\frac12-\kappa$ and $\gamma\le0$, the spatial embedding and the definition of the weighted norm give a constant $C_T$, independent of $\e$, such that
\[
\|\widehat v^\e\|_{\CC^{-1/2-\kappa,T}_{\gamma,\e}}
\le C_T\|\widehat v^\e\|_{\CC^{4\kappa,T}_{0,\e}}.
\]
Consequently, for every $\delta>0$,
\begin{align*}
&\left\{\tau_\e\le T\right\}
\cup
\left\{\tau_\e>T,
\ \|\theta^\e;\theta\|_{\CC^{-1/2-\kappa,T}_{\gamma,\e}}>\delta\right\}
\\
&\quad\subset
\left\{\tau_\e\le T\right\}
\cup
\left\{\|\widehat v^\e\|_{\CC^{4\kappa,T}_{0,\e}}>\frac\delta{2C_T}\right\}
\cup
\left\{\|\Psi^\e;\Psi\|_{\CC^{-1/2-\kappa,T}_{\gamma,\e}}>\frac\delta2\right\}.
\end{align*}
The probability of the first two events tends to zero by Corollary~\ref{cor:remainder-and-exit}, and the probability of the third tends to zero by Theorem~\ref{theo:Conv_SHE}.
\end{proof}

\subsection*{Acknowledgements}
I.C. acknowledges support from the European Research Council (ERC) via the Starting Grant SQGT 101116964. 
H.S. acknowledges support from the National Science Foundation via CAREER DMS-2044415.
This work began in the Program ``Probabilistic methods in quantum field theory'' during summer 2025 at the Hausdorff Research Institute for Mathematics, funded by the Deutsche Forschungsgemeinschaft (DFG, German Research Foundation) under Germany's Excellence Strategy -- EXC-2047/1 -- 390685813.
\medskip

\noindent
For the purpose of open access, the authors have applied a CC BY public copyright licence to any author accepted manuscript arising from this submission.

\bibliographystyle{./Martin}
\bibliography{./refs}
\end{document}